\documentclass[10pt]{article}
\usepackage[margin=1in]{geometry}
\usepackage{amsmath,amsthm,amsfonts,amssymb,amscd,mathrsfs}
\usepackage{enumerate}
\usepackage{enumitem}
\usepackage{graphicx}
\usepackage{tikz}
\usepackage{subcaption}
\usepackage{booktabs}
\usepackage{array,tabularx}
\usepackage{bbm}
\usepackage{float}
\usepackage{algorithm}
\usepackage{algpseudocode}
\usepackage{etoolbox}
\usepackage{cite} 
\usepackage[colorlinks=true,allcolors=black]{hyperref}
\usetikzlibrary{arrows.meta,decorations.pathmorphing}
\patchcmd{\abstract}{\small}{}{}{}

\newtheorem{theorem}{Theorem}[section]
\newtheorem{definition}[theorem]{Definition}

\newtheorem{proposition}[theorem]{Proposition}
\newtheorem{lemma}[theorem]{Lemma}
\newtheorem{assumption}[theorem]{Assumption}
\newtheorem{problem}[theorem]{Problem}
\newenvironment{Problem}{\begin{problem}}{\end{problem}}
\theoremstyle{remark}
\newtheorem{remark}[theorem]{Remark}

\newcommand{\R}{{\mathbb R}}
\newcommand{\E}{{\mathbb E}}

\newcommand{\inprod}[2]{\left\langle #1, #2 \right\rangle}

\newcommand{\BW}{\mathrm{BW}}

\DeclareMathOperator{\vecop}{vec}

\allowdisplaybreaks[4]
\title{Structure and Conditioning in Finite-Horizon Continuous-Time Inverse LQR with Unknown Dynamics}

\author{%
Meiling Yu$^{1}$,\quad
Yuan-Hua Ni$^{1}$,\quad
Lei Jiang$^{2}$\\[0.5em]
\small $^{1}$College of Artificial Intelligence,
Nankai University, China\\
\small $^{2}$China North Artificial Intelligence Innovation
Research Institute, China\\[0.3em]
\small
\href{mailto:yumeiling0512@126.com}
{\texttt{yumeiling0512@126.com}};\quad
\href{mailto:yhni@nankai.edu.cn}
{\texttt{yhni@nankai.edu.cn}};\quad
\href{mailto:jianglei@openloong.net}
{\texttt{jianglei@openloong.net}}
}
\date{}

\begin{document}

\maketitle

\begin{abstract}
Inverse Optimal Control (IOC) aims to infer the underlying cost functional of an agent from observations of its expert behavior. This paper studies the finite-horizon continuous-time inverse LQR problem from closed-loop state--input trajectories, where both the system matrices and the quadratic cost are unknown. The finite-horizon terminal condition can generate a time-varying optimal gain. When sufficiently rich, this gain variation provides the information needed to separate the unknown system matrices. We quantify this mechanism through three computable conditioning indices, which measure state richness, gain-variation richness, and injectivity of a structured cost operator. Using these indices, we establish joint identifiability conditions for the inverse problem considered here. Crucially, these conditions guarantee recovery of the ground-truth system matrices $(A,B)$ and the true cost weighting matrices, rather than merely a behaviorally equivalent surrogate. We also develop a conditioning-aware sampled-data reconstruction method that reconstructs the gain $K(\cdot)$ and the closed-loop dynamics matrix $A_c(\cdot)$ from noisy measurements, recovers $(A,B)$ in closed form, and identifies the quadratic weights through a convex semidefinite program. We further establish the non-asymptotic perturbation bounds and the consistency of the full reconstruction method under sub-Gaussian observation noise, with explicit dependence on the same conditioning indices. Numerical experiments support the theory and illustrate the diagnostic value of the conditioning indices. 

\end{abstract}


\section{Introduction}

Inverse optimal control asks whether one can reconstruct, from observed expert behavior, the performance index for which that behavior is optimal. In the linear--quadratic setting this inverse question is especially attractive: the forward problem is analytically transparent, the recovered weights often admit direct control-theoretic interpretation, and the resulting framework connects naturally to imitation learning, behavior modeling, and human--robot interaction \cite{lewis2012optimal,ab2020inverse,adams2022survey,chan2025inverse}. The classical finite-horizon linear quadratic regulator (LQR) problem is well understood when the system dynamics and the quadratic weights are known. In applications, however, the cost is seldom available a priori, and the system model may be uncertain as well. This is precisely the setting in which the inverse LQR is most useful.

The theory of inverse LQR has traditionally been developed under the assumption that the system dynamics are known. Starting from the foundational work of Kalman~\cite{kalman1964linear} and subsequently formalized by Boyd et al.~\cite{boyd1994linear}, this line of research has established fundamental characterizations of optimality and shown that, when both the system model and the optimal feedback law are available, the underlying cost can be recovered through frequency-domain conditions or Linear Matrix Inequalities (LMIs). Subsequent developments have extended this framework to stochastic systems with process and measurement noise~\cite{zhang2019inverse,zhang2022statistically,li2024inverse,karg2024bi}, partially observed systems~\cite{molloy2016discrete}, unknown control horizons~\cite{qu2024control}, non-autonomous LQR problems~\cite{jean2024inverse}, and multi-agent settings~\cite{zhang2023inverse,hallinan2025inverse}. These advances have considerably expanded the scope of inverse optimal control. Nevertheless, they share a common premise: the system matrices are known, or can be identified independently of the inverse problem itself. These works therefore do not directly answer the question studied here: whether one can jointly recover the system dynamics and the quadratic cost from finite-horizon closed-loop trajectory data when the dynamics are unknown.

When the dynamics are not available a priori, existing data-driven approaches to inverse optimal control may be broadly classified into three categories. The first class proceeds directly from optimality conditions and estimates cost parameters from observed trajectories. For example, Qu et al.~\cite{qu20243dioc} recover the cost weighting matrices from input--output data in the behavioral systems framework. Their approach learns cost weights directly from measured input--output trajectories and requires persistently exciting offline data. The second class is motivated by reinforcement learning. Inverse reinforcement Q-learning~\cite{xue2021inverse} updates the cost weighting matrices jointly with policy iteration, but in general it identifies only a reward representation consistent with the observed feedback, rather than uniquely recovering the true parameters. The third class consists of differentiable optimal control formulations, including Pontryagin differentiable programming~\cite{jin2020pontryagin} and its extension in~\cite{cao2025differential}, which provide flexible end-to-end parameter-learning architectures. However, their theoretical guarantees are inherently tied to local optimization arguments. The recent convex data-driven inverse optimal control method of~\cite{garrabe2025convexddioc} likewise emphasizes tractable recovery from data. Closely related work has examined parameter recovery when the dynamics are unknown. In the discrete-time stochastic LQR setting, Geadah et al.~\cite{geadah2024inferring} characterize infinite-horizon ambiguities and show how finite-horizon terminal conditions or additional control observations can improve parameter recovery. Cheng et al.~\cite{cheng2026ddioc} study discrete-time finite-horizon linear quadratic tracking with unknown target states and unknown dynamics, providing identifiability conditions and statistical consistency under observation noise. Here, we study finite-horizon continuous-time inverse LQR, with emphasis on the conditions under which the dynamics and a normalized structured quadratic cost can be jointly reconstructed, and on how these conditions govern sensitivity to sampling and observation errors.

The contributions of this paper are threefold.

First, we establish a structural identifiability theory for finite-horizon
continuous-time inverse LQR with unknown dynamics. The finite-horizon Riccati terminal condition can generate a time-varying gain $K(t)$.
When sufficiently rich, this gain variation separates
the constant open-loop matrices from the closed-loop matrix
$A_c(t)=A-BK(t)$. We introduce three computable conditioning indices
$c_X$, $c_{AB}$, and $c_{QRH}$, which quantify state richness, gain-variation
richness, and injectivity of the normalized structured cost operator,
respectively. Under $\operatorname{rank}(X(0))=n$, $H=\alpha Q$, and the
normalization $\operatorname{tr}(R)=m$, the conditions
$c_{AB}>0$ and $c_{QRH}>0$ guarantee global identifiability of the
normalized true tuple within the structured class.
Moreover, under $Q^\star\succ0$, the condition $c_{QRH}>0$ is
necessary for global identifiability.

Second, we convert the above identifiability decomposition into an implementable
sampled-data reconstruction procedure, called CR-IOC. The method first estimates
$K(t)$ from the closed-loop relation $ U(t)=-K(t)X(t)$, and then reconstructs $A_c(t)$ from local state increments with a Richardson bias-cancellation step~\cite{richardson1911approx}, thereby avoiding direct numerical differentiation. It subsequently recovers $(A,B)$ in closed form through a gain-variation Gramian, and finally recovers $(Q,R,H)$ through the structured Lyapunov equation and a semidefinite program. The algorithm requires no nested solution of optimal control problems. Moreover, the empirical conditioning indices $\hat c_X$, $\hat c_{AB}$, and $\hat c_{QRH}$ arise naturally from the reconstruction steps and serve as diagnostics for data quality and numerical stability.

 Third, we derive non-asymptotic perturbation bounds for all stages of the reconstruction and establish end-to-end consistency under sub-Gaussian observation noise. The conditioning indices \(c_X\), \(c_{AB}\), and \(c_{QRH}\) quantify the sensitivity of the gain-recovery, dynamics-recovery, and cost-recovery stages, respectively. In particular, as any of these indices approaches zero, the bounds predict rapid deterioration of numerical accuracy. This makes explicit how weak structural informativeness translates into statistical and numerical sensitivity.

The relation to existing IOC literature can be summarized as follows:
\begin{itemize}
    \item Qu et al.~\cite{qu20243dioc} recover cost weights directly from input--output data through a behavioral systems representation and model-free KKT conditions, and provide identifiability and perturbation analyses. Their approach uses offline data satisfying a behavioral rank condition. The present paper studies joint reconstruction of $(A,B,Q,R,H)$ from optimal state--input trajectories within the structured class $H=\alpha Q$, with known $\alpha$ and the normalization $\operatorname{tr}(R)=m$.

    \item Garrab{\'e} et al.~\cite{garrabe2025convexddioc} develop convex cost-reconstruction methods for nonlinear, non-stationary, and stochastic systems. Their inverse algorithm takes probabilistic dynamics descriptions as inputs, which may themselves be estimated from data. Here, the focus is on conditions under which optimal closed-loop trajectories determine both the unknown dynamics and the normalized structured cost in a continuous-time LQR setting.

    \item The discrete-time studies of Geadah et al.~\cite{geadah2024inferring} and Cheng et al.~\cite{cheng2026ddioc} address parameter recovery with unknown dynamics in finite-horizon settings. In the continuous-time problem considered here, the separation of the constant matrices $(A,B)$ is characterized through the relation $A_c(t)=A-BK(t)$ and an integral gain-variation Gramian. The reconstruction analysis also accounts explicitly for errors introduced by sampling continuous trajectories.

    \item Geadah et al.~\cite{geadah2024inferring} identify the propagation of closed-loop estimation errors to system and control parameter recovery as a direction for further analysis. The present work tracks errors from temporal smoothing through recovery of $K(\cdot)$ and $A_c(\cdot)$, separation of $(A,B)$, and recovery of the structured cost. The corresponding conditioning indices quantify error amplification at each stage.
\end{itemize}

The rest of the paper is organized as follows. Section~2 introduces the preliminaries, the problem formulation, and the conditioning-index analysis. Section~3 develops the identifiability theory for the continuous-time setting, including necessary and sufficient conditions for the identifiability of the dynamics and the cost, as well as a joint identifiability theorem. Section~4 presents the conditioning-aware sampled-data reconstruction method. Section~5 establishes the non-asymptotic perturbation bounds and proves the consistency. Section~6 reports numerical experiments that support the theory and illustrate the diagnostic value of the proposed conditioning indices. Section~7 concludes the paper and discusses future directions.

\noindent\textit{Notation.}
$\R^n$ and $\R^{n\times m}$ denote the spaces of real $n$-vectors and real $n\times m$ matrices, respectively.
We write $\mathbb S^n:=\{M\in\R^{n\times n}:M=M^\top\}$ for the space of real symmetric matrices.
For symmetric matrices $M,N\in\mathbb S^n$, the notation $M\succeq N$ ($M\succ N$) means that $M-N$ is positive semidefinite (positive definite).
For a matrix-valued function $W:[0,T]\to\R^{m\times n}$, we denote $\|W\|_{L^2([0,T])}^2:=\int_0^T\|W(t)\|_F^2\,dt$.
For matrices $X,Y$ of matching dimensions, $\inprod{X}{Y}:=\mathrm{tr}(X^\top Y)$.
We use $\|\cdot\|$ for the Euclidean norm and $\|\cdot\|_F$ for the Frobenius norm. For a matrix pair $(Q,R)$, we write
$\|(Q,R)\|_F^2:=\|Q\|_F^2+\|R\|_F^2$.
The notation $a\asymp b$ means that $a/b$ stays bounded above and below by positive constants that are independent of the asymptotic parameters.
For a matrix $M\in\mathbb{R}^{p\times q}$, $\operatorname{vec}(M)\in\mathbb{R}^{pq}$ denotes the column-wise vectorization obtained by stacking the columns of $M$. For a symmetric matrix $S\in\mathbb{S}^{n}$, $\operatorname{vech}(S)\in\mathbb{R}^{n(n+1)/2}$ denotes the half-vectorization obtained by stacking the lower-triangular entries of $S$, including the diagonal entries, in column-wise order.

\section{Preliminaries and Problem formulation}\label{sec:setup}
\subsection{Finite-horizon LQR setup and closed-loop representation}\label{sec:operator-lyap}
We consider $N$ optimal closed-loop state--input trajectories generated over a fixed horizon $[0,T]$ by a finite-horizon continuous-time LQR law.
For each trajectory $i\in\{1,\dots,N\}$, the controlled system is described by
\begin{equation}\label{eq:dyn}
\dot x_i(t)=Ax_i(t)+Bu_i(t),\quad x_i(0)=x_{i,0},\quad t\in[0,T],
\end{equation}
where $A\in\R^{n\times n}$ and $B\in\R^{n\times m}$ are constant and unknown. The trajectories are generated from the different initial conditions $x_{i,0}$, which are later assumed to be sufficiently rich in the sense that the initial state matrix has full row rank.
For compactness, we define
\begin{equation}\label{eq:XU-def}
X(t):=\big[x_1(t)\ \cdots\ x_N(t)\big]\in\R^{n\times N},\qquad
U(t):=\big[u_1(t)\ \cdots\ u_N(t)\big]\in\R^{m\times N}.
\end{equation}
The associated optimal control problem is
\begin{equation}\label{eq:lqr}
\min_{u(\cdot)}\ J(u):=x(T)^\top Hx(T)+\int_0^T\big(x(t)^\top Qx(t)+u(t)^\top Ru(t)\big)\,dt
\end{equation}
with $Q=Q^\top\succeq 0$, $R=R^\top\succ 0$, and $H=H^\top\succeq 0$.
By the classical finite-horizon LQR theory\cite{anderson2007optimal}, the optimal control is unique and has the linear feedback form
\begin{equation}\label{eq:opt-gain}
u(t)=-K(t)x(t),\qquad K(t)=R^{-1}B^\top P(t),
\end{equation}
where $P(\cdot)$ is the unique symmetric solution of the Riccati differential equation
\begin{equation}\label{eq:DRE}
-\dot P(t)=A^\top P(t)+P(t)A-P(t)BR^{-1}B^\top P(t)+Q,\qquad P(T)=H.
\end{equation}

Because the gain \eqref{eq:opt-gain} is invariant under the positive rescaling $(Q,R,H)\mapsto(\gamma Q,\gamma R,\gamma H)$ for any $\gamma>0$, closed-loop trajectories can determine at most the relative scale of the cost weighting matrices, but not their common absolute scale.
We remove this intrinsic ambiguity by imposing
\begin{equation}\label{eq:scale-norm}
\mathrm{tr}(R)=m.
\end{equation}

\begin{assumption}[Standing assumptions]\label{ass:standing}
\hfill
\begin{enumerate}[label=(A\arabic*),leftmargin=*]
\item\label{ass:stab} The pair $(A,B)$ is stabilizable.
\item\label{ass:Brank} The input matrix has full column rank, $\mathrm{rank}(B)=m$, and $m\le n$.
\item\label{ass:richIC} The initial state matrix has full row rank, $\mathrm{rank}(X(0))=n$.
\end{enumerate}
\end{assumption}

\begin{assumption}\label{ass:terminal-prop}
There exists a known scalar $\alpha\ge 0$ such that the terminal weight is proportional to the running state weight
\begin{equation}\label{eq:H-alphaQ}
H=\alpha Q.
\end{equation}
\end{assumption}

\begin{remark}\label{rem:assumptions-common}
Assumptions~\ref{ass:standing} and \ref{ass:terminal-prop} are standard in finite-horizon LQR and inverse LQR.
The stabilizability of $(A, B)$ in Assumption~\ref{ass:standing} ensures that the underlying LQR problem is well-posed\cite{anderson2007optimal}. The nondegenerate input channels and trajectory-rich data matrices are imposed to exclude structural ambiguity in the recovery problem; similar conditions appear in related inverse-LQR formulations\cite{zhang2019lqr,yu2021sysid,li2020ctiqoc,cao2025inverse}.
The proportional terminal penalty in Assumption~\ref{ass:terminal-prop} is also common in finite-horizon LQR/MPC design when a structured terminal cost is adopted \cite{mayne2000mpc,rawlings2009mpc}.
\end{remark}

Since the closed-loop state transition matrix is nonsingular on $[0,T]$, Assumption~\ref{ass:richIC} implies $\mathrm{rank}(X(t))=n$ for all $t\in[0,T]$.
This will guarantee the positivity of the state-richness index introduced below.
We also use the closed-loop dynamics matrix
\begin{equation}\label{eq:Acl}
A_c(t):=A-BK(t),
\end{equation}
so that $\dot x(t)=A_c(t)x(t)$ and $u(t)=-K(t)x(t)$. Using the stationarity condition $B^\top P(t)=RK(t)$, the Riccati differential equation \eqref{eq:DRE} can be rewritten as the terminal-value Lyapunov equation
\begin{equation}\label{eq:CL-lyap}
-\dot P(t)=A_c(t)^\top P(t)+P(t)A_c(t)+Q+K(t)^\top RK(t),\qquad P(T)=H.
\end{equation}
Crucially, \eqref{eq:CL-lyap} is affine in $(Q,R,H)$ and becomes linear when $(A,B,K(\cdot))$ are fixed. The next lemma provides a convenient integral representation for terminal-value Lyapunov equations such as \eqref{eq:CL-lyap}.
\begin{lemma}\label{lem:terminal-lyap}
Let $A_c(\cdot)$ be continuous on $[0,T]$. For $\tau\ge t$, let
$\Phi(\tau,t)$ denote the nonsingular state transition matrix associated with
$\dot z=A_c(t)z$, i.e.,
\[
\frac{\partial}{\partial \tau}\Phi(\tau,t)=A_c(\tau)\Phi(\tau,t),
\qquad
\Phi(t,t)=I.
\]
Consider the terminal-value Lyapunov equation
\[
-\dot P(t)=A_c(t)^\top P(t)+P(t)A_c(t)+S(t),
\qquad
P(T)=H,
\]
where $S(\cdot)\in L^1([0,T];\mathbb R^{n\times n})$ and $H\in\mathbb R^{n\times n}$ is given. Then, this equation admits a unique absolutely continuous solution, satisfying the differential equation for almost every $t\in[0,T]$, i.e.,
\[
P(t)=\Phi(T,t)^\top H\Phi(T,t)
+\int_t^T\Phi(\tau,t)^\top S(\tau)\Phi(\tau,t)\,d\tau .
\]
\end{lemma}

\begin{proof}
Lemma \ref{lem:terminal-lyap} is a standard consequence of the variation-of-constants formula for linear time-varying systems. We include the proof for completeness. Since $A_c(\cdot)$ is continuous on the compact interval $[0,T]$, the state
transition matrix $\Phi(\tau,t)$ is well defined, nonsingular, and satisfies
\[
\frac{\partial}{\partial \tau}\Phi(\tau,t)=A_c(\tau)\Phi(\tau,t),
\qquad
\frac{\partial}{\partial t}\Phi(\tau,t)=-\Phi(\tau,t)A_c(t).
\]
Moreover, if $M_A:=\sup_{s\in[0,T]}\|A_c(s)\|$, then
\[
\|\Phi(\tau,t)\|\le e^{M_A(\tau-t)},\qquad 0\le t\le \tau\le T.
\]
We first prove the existence by defining
\[
P(t):=\Phi(T,t)^\top H\Phi(T,t)
+\int_t^T\Phi(\tau,t)^\top S(\tau)\Phi(\tau,t)\,d\tau .
\]
The integral is well defined because
\[
\|\Phi(\tau,t)^\top S(\tau)\Phi(\tau,t)\|
\le e^{2M_A T}\|S(\tau)\|,
\]
and $S\in L^1([0,T])$. The first term of $P(t)$ is continuously differentiable in $t$,
and the second term is absolutely continuous in $t$ by the Leibniz rule for
parameter-dependent integrals with an $L^1$ integrand. Hence, $P(\cdot)$ is
absolutely continuous. Let
\[
P_H(t):=\Phi(T,t)^\top H\Phi(T,t),
\qquad
P_S(t):=\int_t^T\Phi(\tau,t)^\top S(\tau)\Phi(\tau,t)\,d\tau .
\]
Using $\partial_t\Phi(\tau,t)=-\Phi(\tau,t)A_c(t)$, we obtain
\[
\dot P_H(t)
=
-A_c(t)^\top P_H(t)-P_H(t)A_c(t).
\]
For the integral term, the same identity and the Leibniz rule give, for almost
every $t\in[0,T]$,
\[
\begin{aligned}
\dot P_S(t)
&=
-\Phi(t,t)^\top S(t)\Phi(t,t)
+\int_t^T \frac{\partial}{\partial t}
\left[
\Phi(\tau,t)^\top S(\tau)\Phi(\tau,t)
\right]d\tau \\
&=
-S(t)
-\int_t^T
\left[
A_c(t)^\top \Phi(\tau,t)^\top S(\tau)\Phi(\tau,t)
+\Phi(\tau,t)^\top S(\tau)\Phi(\tau,t)A_c(t)
\right]d\tau \\
&=
-S(t)-A_c(t)^\top P_S(t)-P_S(t)A_c(t).
\end{aligned}
\]
Combining the two displays yields, for almost every $t\in[0,T]$,
\[
\dot P(t)
=
-A_c(t)^\top P(t)-P(t)A_c(t)-S(t).
\]
Also, we have
\[
P(T)=\Phi(T,T)^\top H\Phi(T,T)+\int_T^T\Phi(\tau,T)^\top S(\tau)\Phi(\tau,T)\,d\tau=H.
\]
Thus, the formula defines an absolutely continuous solution.

It remains to prove the uniqueness. Let $\widetilde P(\cdot)$ be another absolutely
continuous solution with the same terminal condition, and set $D(t):=\widetilde P(t)-P(t).$
Then, $D(T)=0$ and, for almost every $t$,
\[
-\dot D(t)=A_c(t)^\top D(t)+D(t)A_c(t).
\]
Fix any $t\in[0,T]$ and define
\[
Y(\tau):=\Phi(\tau,t)^\top D(\tau)\Phi(\tau,t), \quad \forall \tau\in[t,T].
\]
Since $D$ is absolutely continuous and $\Phi(\cdot,t)$ is continuously
differentiable, $Y$ is absolutely continuous. Differentiating for almost every
$\tau\in[t,T]$ gives
\[
\begin{aligned}
\dot Y(\tau)=
\Phi(\tau,t)^\top
\left[
A_c(\tau)^\top D(\tau)+\dot D(\tau)+D(\tau)A_c(\tau)
\right]
\Phi(\tau,t)=0.
\end{aligned}
\]
Therefore, $Y(\tau)$ is constant on $[t,T]$. Hence,
\[
D(t)=Y(t)=Y(T)=\Phi(T,t)^\top D(T)\Phi(T,t)=0.
\]
Since $t\in[0,T]$ is arbitrary, $D(t)=0$ for all $t\in[0,T]$. Thus,
$\widetilde P=P$, and the solution is unique.
\end{proof}

\subsection{Linear operator formulation and conditioning indices}\label{sec:operator-linear}
With the finite-horizon LQR setup fixed, the inverse problem studied in this paper can be stated as follows.

\begin{Problem}[Finite-horizon continuous-time inverse LQR]\label{prob:inverse-lqr}
Consider a set of closed-loop state--input trajectories $\mathcal D:=\{(x_i(\cdot),u_i(\cdot))\}_{i=1}^N$ generated over the horizon $[0,T]$ by the optimal control law \eqref{eq:opt-gain} applied to the unknown linear system \eqref{eq:dyn}. Under Assumptions~\ref{ass:standing} and \ref{ass:terminal-prop} together with the normalization \eqref{eq:scale-norm}, recover the unknown parameter tuple $\Theta=(A,B,Q,R,H)$ with $Q\succeq 0$, $R\succ 0$, and $H=\alpha Q$, such that the finite-horizon LQR problem determined by $\Theta$ reproduces the optimal closed-loop behavior exhibited in $\mathcal D$.
\end{Problem}


\begin{definition}\label{def:operators}
Fix matrices $(A,B)$ and a continuous gain function
$K:[0,T]\to\mathbb{R}^{m\times n}$. For any symmetric triple
$(Q,R,H)$, let $P_{Q,R,H}(\cdot)$ denote the unique solution of
\begin{equation}\label{eq:P-operator}
-\dot P=A_c^\top P+PA_c+Q+K^\top RK,\qquad P(T)=H.
\end{equation}
The associated linear stationarity residual operator $\mathscr{M}$ is defined pointwise by
\begin{equation}\label{eq:M-operator}
\big(\mathscr{M}(Q,R,H)\big)(t):=B^\top P_{Q,R,H}(t)-RK(t),\qquad t\in[0,T].
\end{equation}
\end{definition}

Linearity of $P_{Q,R,H}$ follows from Lemma~\ref{lem:terminal-lyap} because \eqref{eq:P-operator} is affine in $(Q,R,H)$ and becomes linear when $(A,B,K(\cdot))$ are fixed; linearity of $\mathscr{M}$ is immediate from \eqref{eq:M-operator}.

\begin{definition}[Conditioning indices]\label{def:indices}
The following quantities will be used to state identifiability conditions and to quantify numerical conditioning.
\begin{enumerate}[label=(\roman*),leftmargin=*]
\item\label{it:cX} For the noiseless state matrix $X(t)$ in \eqref{eq:XU-def}, define
\begin{equation}\label{eq:cX}
c_X:=\inf_{t\in[0,T]}\lambda_{\min}\big(X(t)X(t)^\top\big).
\end{equation}
\item\label{it:cAB} Let
\begin{equation}\label{eq:K-mean}
\bar K:=\frac{1}{T}\int_0^T K(t)\,dt,\qquad \Delta K(t):=K(t)-\bar K,
\end{equation}
and define the Gramian
\begin{equation}\label{eq:GK}
\mathcal{G}_K:=\int_0^T\Delta K(t)\Delta K(t)^\top\,dt,\qquad c_{AB}:=\lambda_{\min}(\mathcal{G}_K).
\end{equation}
\item\label{it:cQRH} Under Assumption~\ref{ass:terminal-prop} with fixed $\alpha$, the cost is parameterized by $(Q,R)$ with $H=\alpha Q$.
Define the structured stationarity residual operator
\begin{equation}\label{eq:Malpha-def}
\big(\mathscr{M}_\alpha(Q,R)\big)(t):=\big(\mathscr{M}(Q,R,\alpha Q)\big)(t)
= B^\top P_{Q,R,\alpha Q}(t)-RK(t),\qquad t\in[0,T],
\end{equation}
where $P_{Q,R,\alpha Q}(\cdot)$ solves \eqref{eq:P-operator} with terminal condition $P(T)=\alpha Q$.
Let
\begin{equation*}
\mathcal{V}_\alpha:=\Big\{(\Delta Q,\Delta R):\ \Delta Q=\Delta Q^\top,\ \Delta R=\Delta R^\top,\ \mathrm{tr}(\Delta R)=0\Big\}
\end{equation*}
be the difference space of the normalized structured cost family $H=\alpha Q$ and $\mathrm{tr}(R)=m$. Indeed, if two structured normalized pairs $(Q_1,R_1)$ and $(Q_2,R_2)$ satisfy $\mathrm{tr}(R_1)=\mathrm{tr}(R_2)=m$, then their difference belongs to $\mathcal V_\alpha$, and the associated terminal perturbation is $\Delta H=\alpha\Delta Q$. we define
\begin{equation}\label{eq:cQRH}
c_{QRH}:=\inf_{(\Delta Q,\Delta R)\in\mathcal{V}_\alpha\,:\,\|(\Delta Q,\Delta R)\|_F=1}\ \big\|\mathscr{M}_\alpha(\Delta Q,\Delta R)\big\|_{L^2([0,T])}.
\end{equation}
\end{enumerate}
\end{definition}

Under Assumption~\ref{ass:richIC}, the noiseless index $c_X$ is automatically positive, while its empirical counterpart remains an important sampled-data diagnostic. The quantity $c_{QRH}$ is the restricted minimum gain of the structured stationarity operator after removing the intrinsic scaling direction. Equivalently, it measures how strongly any unit-size admissible perturbation $(\Delta Q,\Delta R)$ is detected by the residual $B^\top P(t)-RK(t)$. Thus, $c_{QRH}>0$ means that no nonzero normalized structured cost perturbation can leave the stationarity residual identically zero. More importantly, the three indices separate the inverse problem into three mechanisms: recovery of the gain from state richness, separation of $(A,B)$ through gain variation, and injectivity of the structured cost map after scale removal. This decomposition will serve as the organizing principle for the identifiability analysis, the sampled-data reconstruction method, and the perturbation bounds developed in the following sections.

\section{Identifiability and conditioning}\label{sec:identifiability}
This section establishes the continuous-time noiseless identifiability theory for finite-horizon inverse LQR. Let $\Theta^\star:=(A^\star,B^\star,Q^\star,R^\star,H^\star)$
denote the true data-generating tuple. The analysis follows the decomposition introduced in Section~\ref{sec:setup}: state richness determines the recovery of $K^\star(\cdot)$ and $A_c^\star(\cdot)$, gain variation determines the recovery of $(A^\star,B^\star)$, and the residual operator determines the recovery of the normalized structured cost. Throughout Section~\ref{sec:identifiability}, we work in an idealized noiseless setting: the optimal trajectories are available as continuous functions $(x_i(t),u_i(t))$ on $[0,T]$ and satisfy the deterministic dynamics \eqref{eq:dyn}. Direct observation of time derivatives is not assumed. The arguments rely on the matrix relation $U(t)=-K^\star(t)X(t)$ together with the integral form of the closed-loop dynamics introduced in Section~\ref{sec:setup}. Observation noise and sampled-data reconstruction are deferred to Sections~\ref{sec:algorithm} and \ref{sec:noise}, respectively.

\begin{definition}\label{def:behavioral-equivalence}
Fix the noiseless trajectory set $\mathcal D$.
Two admissible tuples
\[
\Theta=(A,B,Q,R,H),\qquad \bar\Theta=(\bar A,\bar B,\bar Q,\bar R,\bar H)
\]
are said to be trajectory-equivalent on $\mathcal D$ if they generate the same state--input trajectories for all data in $\mathcal D$.
The identifiable object of the inverse problem is therefore the trajectory-equivalence class $[\Theta]_{\mathcal D}$.
Unique parameter identification means that this class reduces to a singleton after removing the intrinsic scale ambiguity \eqref{eq:scale-norm}.
\end{definition}
We restrict attention to candidate tuples $(\bar A,\bar B,\bar Q,\bar R,\bar H)$ that define a well-posed LQR problem of the form \eqref{eq:dyn}--\eqref{eq:lqr}, i.e., $\bar Q\succeq 0$, $\bar R\succ 0$, $\bar H\succeq 0$, together with the standing assumptions in Assumption~\ref{ass:standing}.
We say that the true tuple $\Theta^\star$ is globally identifiable from $\mathcal{D}$ within the prescribed admissible class if every admissible tuple $\bar\Theta$ that is trajectory-equivalent to $\Theta^\star$ on $\mathcal D$ satisfies $\bar\Theta=\Theta^\star$.

\subsection{Finite-horizon identifiability of system model}\label{sec:ident-finite}
We first characterize when the open-loop dynamics can be identified from finite-horizon optimal trajectories.
The analysis is constructive: we identify $K^\star(\cdot)$ from the algebraic closed-loop relation, infer the closed-loop dynamics matrix $A_c^\star(\cdot)$ from the state trajectories, and then recover constant $(A^\star,B^\star)$ from the functional identity $A_c^\star(t)=A^\star-B^\star K^\star(t)$.

\begin{lemma}\label{lem:K-ident}
Assume the noiseless trajectory data satisfy $U(t)=-K^\star(t)X(t)$ for all $t\in[0,T]$.
If $c_X>0$ in \eqref{eq:cX}, then $K^\star(t)$ is uniquely determined by
\begin{equation}\label{eq:K-formula}
K^\star(t)=-U(t)X(t)^\top\big(X(t)X(t)^\top\big)^{-1},\quad \forall t\in[0,T].
\end{equation}
\end{lemma}

\begin{proof}
Fix $t$ and suppose $X(t)X(t)^\top$ is invertible.
From $U(t)=-K^\star(t)X(t)$, we right-multiply by $X(t)^\top\big(X(t)X(t)^\top\big)^{-1}$ to obtain
\begin{equation*}
U(t)X(t)^\top\big(X(t)X(t)^\top\big)^{-1}=-K^\star(t)X(t)X(t)^\top\big(X(t)X(t)^\top\big)^{-1}=-K^\star(t),
\end{equation*}
which yields \eqref{eq:K-formula}.
Uniqueness follows because the right-hand side is determined solely by $(X(t),U(t))$.
Together with $c_X>0$, this implies $\lambda_{\min}(X(t)X(t)^\top)\ge c_X>0$ for all $t$, and thus the inversion is valid on the entire horizon.
\end{proof}

The same conclusion can be expressed through the closed-loop integral relation
\[
X(t)-X(0)=\int_0^t A_c^\star(\tau)X(\tau)\,d\tau,\qquad t\in[0,T].
\]
Under the rank condition $\mathrm{rank}(X(t)X(t)^{\top})=\mathrm{rank}(X(t))=n$, this relation determines $A_c^\star(\cdot)$ uniquely. When written in differential form, the corresponding pointwise identity is
\begin{equation}\label{eq:Ac-direct}
A_c^\star(t)=\dot X(t)X(t)^\top\big(X(t)X(t)^\top\big)^{-1}.
\end{equation}
In the noiseless setting, \eqref{eq:Ac-direct} holds almost everywhere. In practice, we avoid numerical differentiation and estimate $A_c^\star(\cdot)$ from state increments (Section~\ref{sec:alg-closedloop}).

We now characterize when the constant matrices $(A^\star,B^\star)$ are uniquely determined by the functional relation $A_c^\star(t)=A^\star-B^\star K^\star(t)$.
The finite-horizon optimal gain $K^\star(t)$ can vary with time; sufficiently rich gain variation provides the information quantified by $c_{AB}$.

\begin{theorem}\label{thm:AB-ident}
Suppose $K^\star(\cdot)$ is continuous and that $K^\star(\cdot)$ and $A_c^\star(\cdot)$ are known on $[0,T]$.
Define
\[
\bar K^\star:=\frac{1}{T}\int_0^T K^\star(t)\,dt,
\qquad
\Delta K^\star(t):=K^\star(t)-\bar K^\star,
\qquad
\mathcal G_{K^\star}:=\int_0^T\Delta K^\star(t)\Delta K^\star(t)^\top\,dt,
\]
\[
\bar A_c^\star:=\frac{1}{T}\int_0^T A_c^\star(t)\,dt,
\qquad
\Delta A_c^\star(t):=A_c^\star(t)-\bar A_c^\star.
\]
Then, $c_{AB}=\lambda_{\min}(\mathcal G_{K^\star})>0$ if and only if
$(A^\star,B^\star)$ is the unique constant pair $(A,B)$ satisfying
\[
A_c^\star(t)=A-BK^\star(t),\qquad \forall t\in[0,T].
\]
When this condition holds, the pair is recovered by
\begin{equation}\label{eq:B-formula}
B=-\Big(\int_0^T\Delta A_c^\star(t)\,\Delta K^\star(t)^\top\,dt\Big)\,\mathcal{G}_{K^\star}^{-1},
\qquad
A=\bar A_c^\star+B\bar K^\star.
\end{equation}
\end{theorem}

\begin{proof}
Averaging $A_c^\star(t)=A^\star-B^\star K^\star(t)$ over $t\in[0,T]$ yields
\begin{equation}\label{eq:mean-Ac}
\bar A_c^\star = A^\star-B^\star\bar K^\star.
\end{equation}
Subtracting \eqref{eq:mean-Ac} from $A_c^\star(t)$ gives, for all $t$,
\begin{equation}\label{eq:delta-Ac}
\Delta A_c^\star(t)= -B^\star\Delta K^\star(t).
\end{equation}
Multiplying \eqref{eq:delta-Ac} on the right by $\Delta K^\star(t)^\top$ and integrating over $[0,T]$ yield
\begin{equation}\label{eq:int-identity}
\int_0^T\Delta A_c^\star(t)\,\Delta K^\star(t)^\top\,dt
=-B^\star\int_0^T\Delta K^\star(t)\Delta K^\star(t)^\top\,dt
=-B^\star\mathcal{G}_{K^\star}.
\end{equation}

Sufficiency.
Assume that $c_{AB}>0$.
Then, $\mathcal{G}_{K^\star}$ is invertible, and \eqref{eq:int-identity} yields the first formula in \eqref{eq:B-formula}.
The second formula follows from \eqref{eq:mean-Ac}.
Now, let $(A,B)$ be any constant pair satisfying $A_c^\star(t)=A-BK^\star(t)$ for all $t$.
Repeating the same averaging and centering argument gives
\[
\Delta A_c^\star(t)=-B\Delta K^\star(t),\qquad
\int_0^T\Delta A_c^\star(t)\,\Delta K^\star(t)^\top\,dt=-B\mathcal{G}_{K^\star}.
\]
Since $\mathcal{G}_{K^\star}$ is invertible, we obtain $B=B^\star$, and then $A=A^\star$ from the averaged identity.
Hence, $(A^\star,B^\star)$ is the unique constant pair consistent with $(A_c^\star(\cdot),K^\star(\cdot))$.

Necessity.
Assume that $c_{AB}=0$.
Then, it holds $\mathcal{G}_{K^\star}$ is singular. Hence, there exists a nonzero vector $p\in\R^m$ such that $p^\top\mathcal{G}_{K^\star}p=0$.
By the definition of $\mathcal{G}_{K^\star}$,
\[
0=p^\top\mathcal{G}_{K^\star}p=\int_0^T\|\Delta K^\star(t)^\top p\|_2^2\,dt,
\]
which implies $\Delta K^\star(t)^\top p=0$ for almost all $t$.
Since $K^\star(\cdot)$ is continuous, the map $t\mapsto \Delta K^\star(t)^\top p$ is continuous and must vanish for all $t\in[0,T]$.
Fix any nonzero vector $b\in\R^n$ and define $E:=bp^\top\in\R^{n\times m}$.
Then, for all $t$,
\[
E\Delta K^\star(t)=b\,p^\top\Delta K^\star(t)=0.
\]
Now, define
\begin{equation}\label{eq:AB-perturb}
B:=B^\star+E,\qquad A:=A^\star+E\bar K^\star.
\end{equation}
Using $K^\star(t)=\bar K^\star+\Delta K^\star(t)$ and $E\Delta K^\star(t)=0$, we compute
\begin{align*}
A-BK^\star(t)
&=\big(A^\star+E\bar K^\star\big)-\big(B^\star+E\big)\big(\bar K^\star+\Delta K^\star(t)\big)\\
&=A^\star-B^\star\bar K^\star-B^\star\Delta K^\star(t)+E\bar K^\star-E\bar K^\star-E\Delta K^\star(t)\\
&=A^\star-B^\star K^\star(t)=A_c^\star(t)
\end{align*}
for all $t$.
Because $E\neq 0$, the pair $(A,B)$ is distinct from $(A^\star,B^\star)$.
Hence, when $c_{AB}=0$, the true pair is not uniquely determined by $(A_c^\star(\cdot),K^\star(\cdot))$.
This proves the equivalence.
\end{proof}

\begin{remark}[Interpretation of $c_{AB}$]
The quantity $c_{AB}=\lambda_{\min}(\mathcal{G}_{K^\star})$ measures the directional richness of the gain variation, not only its magnitude.
When $m=1$, the condition $c_{AB}>0$ reduces to $K^\star(\cdot)$ being nonconstant on $[0,T]$.
When $m>1$, time variation alone is not enough: the row-valued variations $\Delta K^\star(t)$ must excite every input direction in $L^2([0,T])$.
Hence, a small empirical value $\hat c_{AB}$ can arise either because $K^\star(\cdot)$ is nearly time invariant, or because its variation is confined to a lower-dimensional input subspace.
In practice, when $\hat c_{AB}$ is small, one can (i) restrict reconstruction to a subinterval on which $K^\star(t)$ varies sufficiently, (ii) combine trajectories collected with different horizons or terminal weights.
\end{remark}

\subsection{Cost identifiability and joint recovery}\label{sec:ident-QRH}
Having identified $(A^\star,B^\star)$, we turn to the quadratic weights.
For fixed $(A^\star,B^\star,K^\star(\cdot))$, the closed-loop optimality identities in Section~\ref{sec:setup} induce the linear residual operator $\mathscr{M}$ in \eqref{eq:M-operator}, whose kernel characterizes the full equivalence class of quadratic costs consistent with the observed closed-loop behavior.

We first consider the unstructured cost class in which the terminal weight $H\succeq0$ is allowed to vary independently of $Q$.
When the system is underactuated ($m<n$), there is an intrinsic ``potential-shaping'' ambiguity: one can modify $(Q,H)$ by adding/subtracting a total derivative term in the objective without changing the optimal control law.
This enlarges the kernel of $\mathscr{M}$ and creates an additional source of cost nonuniqueness beyond the positive scaling \eqref{eq:scale-norm}.
We record the construction for completeness.


\begin{proposition}\label{prop:shaping}
Suppose $\mathrm{rank}(B)=m<n$.
Then, there exists a nonzero matrix $S=S^\top\succeq0$ with $SB=0$.
For any such $S$ and any $\tau\in\R$, define the shaped triple
\begin{equation}\label{eq:shaping}
  Q_\tau:=Q+\tau(A^\top S+SA),\qquad R_\tau:=R,\qquad H_\tau:=H-\tau S.
\end{equation}
Then, for every $x_0\in\R^n$,
\begin{equation}\label{eq:shaping-cost}
  J_{Q_\tau,R_\tau,H_\tau}(u)=J_{Q,R,H}(u)-\tau\,x_0^\top S\,x_0.
\end{equation}
In particular, the optimal control and the associated feedback gain are invariant under the shaping~\eqref{eq:shaping}.
Moreover, if $Q\succ0$, then there exists $\delta>0$ such that,
for every $\tau\in(-\delta,0)$, $(Q_\tau,R_\tau,H_\tau)$ is
admissible and distinct from $(Q,R,H)$, while
$(A,B,Q_\tau,R_\tau,H_\tau)$ is trajectory-equivalent to
$(A,B,Q,R,H)$ for every initial state.
\end{proposition}

\begin{proof}
Because $\mathrm{rank}(B)=m<n$, the nullspace of $B^\top$ is nontrivial.
Choose $0\neq v\in\ker(B^\top)$ and set $S:=vv^\top$.
Then, $S=S^\top\succeq0$ and
\[
SB=vv^\top B=v(B^\top v)^\top=0.
\]
Let $x(\cdot)$ be the state trajectory under the control input $u(\cdot)$.
Expanding the objectives gives
\begin{equation}\label{J_1}
    J_{Q_\tau,R_\tau,H_\tau}(u)-J_{Q,R,H}(u)
=
\tau\!\left(\int_0^T x(t)^\top(A^\top S+SA)x(t)\,dt\ -\ x(T)^\top S x(T)\right).
\end{equation}
Since $\dot x=Ax+Bu$, it holds
\[
\frac{d}{dt}\big(x^\top S x\big)
= x^\top(A^\top S+SA)x + 2x^\top SBu.
\]
The condition $SB=0$ implies $x^\top SBu = 0$, hence
\[
\int_0^T x(t)^\top(A^\top S+SA)x(t)\,dt = x(T)^\top S x(T)-x(0)^\top S x(0).
\]
Substituting the preceding expression for the integral term in
\eqref{J_1} yields \eqref{eq:shaping-cost}.
The difference between the two objectives is therefore independent of $u(\cdot)$, thus, the minimizers and the induced optimal feedback law are unchanged.

Now, assume $Q\succ0$ and choose $\tau<0$.
Since $A^\top S+SA$ is symmetric, Weyl's inequality gives
\[
\lambda_{\min}(Q_\tau)
\ge \lambda_{\min}(Q)-|\tau|\,\|A^\top S+SA\|.
\]
Therefore, $Q_\tau\succ0$ whenever $|\tau|\,\|A^\top S+SA\|<\lambda_{\min}(Q)$.
Because $\tau<0$ and $S\succeq0$, one has
\[
H_\tau=H+|\tau|S\succeq0,
\qquad
R_\tau=R\succ0.
\]
Moreover, $H_\tau\neq H$ since $S\neq0$ and $\tau\neq0$, thus, the shaped triple is distinct from $(Q,R,H)$.
Equation~\eqref{eq:shaping-cost} then shows that $(Q_\tau,R_\tau,H_\tau)$ is trajectory-equivalent to $(Q,R,H)$.
This proves the last claim.
\end{proof}

We now return to the structured class in Assumption~\ref{ass:terminal-prop}, with $H=\alpha Q$ and $\operatorname{tr}(R)=m$, and characterize normalized cost uniqueness through $\mathscr{M}_\alpha$ in~\eqref{eq:Malpha-def}.


\begin{theorem}\label{thm:QRH-kernel}
Fix the true system matrices $(A^\star,B^\star)$, the induced
optimal gain $K^\star(\cdot)$, and the known scalar $\alpha\geq0$.
Let $\mathscr{M}_\alpha$ denote the corresponding structured
stationarity residual operator, and define
\begin{equation}\label{eq:Falphanorm}
\mathcal{F}_{\alpha,m}:=\left\{(Q,R)\in\mathbb{S}^n\times\mathbb{S}^m:
Q\succeq0,\ R\succ0,\ \operatorname{tr}(R)=m,\ \mathscr{M}_\alpha(Q,R)\equiv0\right\}.
\end{equation}
Let $(Q^\star,R^\star)\in\mathcal{F}_{\alpha,m}$ be the true
normalized structured cost-weight pair, with
$H^\star=\alpha Q^\star$. Then, the following assertions hold.
\begin{enumerate}[label=(\alph*),leftmargin=*]
\item If $c_{QRH}>0$, then
\[
\mathcal{F}_{\alpha,m}
=
\{(Q^\star,R^\star)\}.
\]
\item If $Q^\star\succ0$, then
\[
\mathcal{F}_{\alpha,m}
=
\{(Q^\star,R^\star)\}
\quad\Longleftrightarrow\quad
c_{QRH}>0.
\]
\end{enumerate}
\end{theorem}

\begin{proof}
Because the unit sphere in $\mathcal{V}_\alpha$ is compact and the map
\[
(\Delta Q,\Delta R)\mapsto \big\|\mathscr{M}_\alpha(\Delta Q,\Delta R)\big\|_{L^2([0,T])}
\]
is continuous, $c_{QRH}>0$ is equivalent to
\[
\ker(\mathscr{M}_\alpha)\cap\mathcal{V}_\alpha=\{(0_{n\times n},0_{m\times m})\}.
\]

Sufficiency.
Assume $c_{QRH}>0$ and let $(Q,R)\in\mathcal{F}_{\alpha,m}$.
Then, $(\Delta Q,\Delta R):=(Q-Q^\star,R-R^\star)$ belongs to $\mathcal{V}_\alpha$ because both pairs satisfy the same trace normalization.
Moreover, by the linearity of $\mathscr{M}_\alpha$,
\[
\mathscr{M}_\alpha(\Delta Q,\Delta R)=\mathscr{M}_\alpha(Q,R)-\mathscr{M}_\alpha(Q^\star,R^\star)\equiv0.
\]
Hence, $(\Delta Q,\Delta R)\in\ker(\mathscr{M}_\alpha)\cap\mathcal{V}_\alpha$.
The trivial-intersection property implied by $c_{QRH}>0$ gives $(\Delta Q,\Delta R)=(0,0)$, namely $(Q,R)=(Q^\star,R^\star)$.
Therefore, $\mathcal{F}_{\alpha,m}=\{(Q^\star,R^\star)\}$.

Necessity under $Q^\star\succ0$.
Assume now that $Q^\star\succ0$ and $c_{QRH}=0$.
Then, there exists a nonzero direction $(\Delta Q,\Delta R)\in\mathcal{V}_\alpha$ such that $\mathscr{M}_\alpha(\Delta Q,\Delta R)\equiv0$.
Since $Q^\star\succ0$ and $R^\star\succ0$, the positive definite cones are open.
Therefore, for all sufficiently small $|\varepsilon|>0$,
\[
Q^\star+\varepsilon\Delta Q\succ0,\qquad R^\star+\varepsilon\Delta R\succ0.
\]
Because $\mathrm{tr}(\Delta R)=0$, the trace normalization is preserved, and the linearity gives
\[
\mathscr{M}_\alpha(Q^\star+\varepsilon\Delta Q,\ R^\star+\varepsilon\Delta R)
=\mathscr{M}_\alpha(Q^\star,R^\star)+\varepsilon\mathscr{M}_\alpha(\Delta Q,\Delta R)\equiv0.
\]
Hence, $(Q^\star+\varepsilon\Delta Q,\ R^\star+\varepsilon\Delta R)\in\mathcal{F}_{\alpha,m}$ for all sufficiently small nonzero $|\varepsilon|$, and these pairs are distinct from $(Q^\star,R^\star)$ because $(\Delta Q,\Delta R)\neq0$.
Thus, $\mathcal{F}_{\alpha,m}$ is not a singleton.
The contrapositive proves part (b).
\end{proof}

Theorem~\ref{thm:QRH-kernel} shows that $c_{QRH}$ is the relevant conditioning index for exact recovery of the normalized structured cost: positivity of $c_{QRH}$ yields uniqueness, and under $Q^\star\succ0$ the same condition is also necessary.
We now combine the previous results into a single joint identifiability statement for the full inverse problem.

\begin{theorem}[Joint identifiability]\label{thm:joint-ident}
Consider noiseless optimal trajectories generated over $[0,T]$ by the
true tuple
\[\Theta^\star=(A^\star,B^\star,Q^\star,R^\star,H^\star)\]
under Assumptions~\ref{ass:standing} and
\ref{ass:terminal-prop}, with $\mathrm{tr}(R^\star)=m$.
If
$c_{AB}>0,\qquad c_{QRH}>0$, then $\Theta^\star$ is globally identifiable from these data within
the admissible class satisfying $H=\alpha Q$ and
$\mathrm{tr}(R)=m$.
Moreover, if $Q^\star\succ0$, then $c_{QRH}>0$ is necessary for
global identifiability within this class.
\end{theorem}

\begin{proof}
Assumption~\ref{ass:richIC}, together with the nonsingularity of the
closed-loop state transition matrix, implies
$\mathrm{rank}(X(t))=n, t\in[0,T]$. Since $X(\cdot)$ is continuous and $[0,T]$ is compact, it follows that
$c_X>0$. Hence, Lemma~\ref{lem:K-ident} applies throughout the
horizon.

\emph{Sufficiency.}
Assume $c_{AB}>0$ and $c_{QRH}>0$, and let
$\bar\Theta=(\bar A,\bar B,\bar Q,\bar R,\bar H)$ be an admissible tuple in the normalized structured class that is
trajectory-equivalent to $\Theta^\star$.
The two tuples generate the same $X(t)$ and $U(t)$.
Therefore, Lemma~\ref{lem:K-ident} implies that they induce the same
gain $K^\star(t)$ for every $t\in[0,T]$.

The true Riccati solution $P^\star(\cdot)$ is continuously
differentiable, and hence $K^\star(\cdot)$ is continuous.
Define
\[
A_c^\star(t):=A^\star-B^\star K^\star(t),\qquad
\bar A_c(t):=\bar A-\bar B K^\star(t).
\]
Since the state trajectories coincide, both tuples satisfy
\[
X(t)-X(0)=\int_0^t A_c^\star(\tau)X(\tau)\,d\tau=\int_0^t \bar A_c(\tau)X(\tau)\,d\tau .
\]
Set $G(t):=\bigl(A_c^\star(t)-\bar A_c(t)\bigr)X(t)$. Then
$\int_0^t G(\tau)\,d\tau=0, t\in[0,T].$ The function $G(\cdot)$ is continuous, so the fundamental theorem of
calculus gives
$G(t)=0, t\in[0,T].$
Since $X(t)$ has full row rank,
\[
A_c^\star(t)-\bar A_c(t)=G(t)X(t)^\top\bigl(X(t)X(t)^\top\bigr)^{-1}=0.
\]
Thus, $A_c^\star(\cdot)=\bar A_c(\cdot).$ Theorem~\ref{thm:AB-ident}, together with $c_{AB}>0$, then yields
\[
(\bar A,\bar B)=(A^\star,B^\star).
\]
With $(A^\star,B^\star,K^\star(\cdot))$ fixed, the closed-loop
Lyapunov and stationarity identities for the two tuples imply
$
(Q^\star,R^\star),
\ (\bar Q,\bar R)
\in
\mathcal F_{\alpha,m}
$. Since $c_{QRH}>0$, Theorem~\ref{thm:QRH-kernel}(a) gives
\[
(\bar Q,\bar R)=(Q^\star,R^\star).
\]
Finally, $H=\alpha Q$ yields
$\bar H=H^\star$. Therefore, $\bar\Theta=\Theta^\star$, proving global
identifiability.

\emph{Necessary.}
Suppose now that $Q^\star\succ0$ and $c_{QRH}=0$.
By Theorem~\ref{thm:QRH-kernel}(b), there exists
$(\widetilde Q,\widetilde R)\in\mathcal F_{\alpha,m}\setminus\{(Q^\star,R^\star)\}$. Set
$\widetilde H:=\alpha\widetilde Q$ and let
$\widetilde P:=P_{\widetilde Q,\widetilde R,\widetilde H}.$ By the definition of $\mathcal F_{\alpha,m}$,
\[
-\dot{\widetilde P}(t)=
A_c^\star(t)^\top\widetilde P(t)
+
\widetilde P(t)A_c^\star(t)
+
\widetilde Q
+
K^\star(t)^\top\widetilde R K^\star(t), \qquad
\widetilde P(T)=\widetilde H,
\]
and
\[
B^{\star\top}\widetilde P(t)
=
\widetilde R K^\star(t).
\]
Using $A_c^\star(t)=A^\star-B^\star K^\star(t),$ the preceding identities give
\[
\begin{aligned}
-\dot{\widetilde P}(t)
&=
A^{\star\top}\widetilde P(t)
+
\widetilde P(t)A^\star
-
\widetilde P(t)B^\star
\widetilde R^{-1}
B^{\star\top}\widetilde P(t)
+
\widetilde Q, \\
K^\star(t)
&=
\widetilde R^{-1}
B^{\star\top}\widetilde P(t).
\end{aligned}
\]
Hence, $(A^\star,B^\star,\widetilde Q,\widetilde R,\widetilde H)
$ is an admissible LQR tuple whose unique optimal gain is
$K^\star(\cdot)$. Since its dynamics are also
$(A^\star,B^\star)$, it generates the same state--input trajectories
as $\Theta^\star$.

The two tuples are distinct because
\[
(\widetilde Q,\widetilde R)
\neq
(Q^\star,R^\star).
\]
Therefore, $\Theta^\star$ is not globally identifiable when
$c_{QRH}=0$. This proves that $c_{QRH}>0$ is necessary under
$Q^\star\succ0$.
\end{proof}

\section{Conditioning-aware estimation from sampled noisy trajectory data}\label{sec:algorithm}
This section converts the continuous-time identifiability decomposition of Section~\ref{sec:identifiability} into the proposed sampled-data reconstruction method built from noisy state-input measurements. The stages follow the theory: temporal denoising, gain reconstruction, local reconstruction of the closed-loop dynamics matrix, recovery of the constant pair $(A^\star,B^\star)$, and structured cost recovery. The empirical diagnostics $(\hat c_X,\hat c_{AB},\hat c_{QRH})$ are computed alongside these stages. Classical Richardson extrapolation is used only in the local reconstruction of $A_c^\star(\cdot)$, where it serves as a bias-cancellation device for sampled increments.

\subsection{Sampling model and temporal denoising}\label{sec:noise-assumptions}
We observe the trajectories at sampling times $t_k:=k\Delta t$, $k=0,1,\dots,L$, with $T=L\Delta t$.
For each trajectory $i$ and sample index $k$, set
\begin{equation}\label{eq:obs-noise}
\tilde x_{i,k}=x_i(t_k)+v^x_{i,k},\qquad \tilde u_{i,k}=u_i(t_k)+v^u_{i,k},
\end{equation}
where $v^x_{i,k}\in\R^n$ and $v^u_{i,k}\in\R^m$ are observation noises.
We write
\begin{equation*}
\tilde X_k:=[\tilde x_{1,k}\ \cdots\ \tilde x_{N,k}]\in\R^{n\times N},\qquad
\tilde U_k:=[\tilde u_{1,k}\ \cdots\ \tilde u_{N,k}]\in\R^{m\times N}.
\end{equation*}
All statements below are non-asymptotic. 

\begin{assumption}[Observation noise]\label{ass:noise}
The observation noises in \eqref{eq:obs-noise} satisfy:
\begin{enumerate}[label=(B\arabic*),leftmargin=*]
\item\label{ass:noise-mean} $\E[v^x_{i,k}]=0$, $\E[v^u_{i,k}]=0$, $i=1,2,\cdots N,k=0,1,\cdots L$.
\item\label{ass:noise-indep} The family $\{(v^x_{i,k},v^u_{i,k})\}_{i,k}$ is independent across $(i,k)$ and independent of the noiseless trajectories $\{(x_i(t_k),u_i(t_k))\}_{i,k}$.
\item\label{ass:noise-subg} For some parameters $\sigma_x,\sigma_u>0$, each $v^x_{i,k}$ is $\sigma_x$-sub-Gaussian and each $v^u_{i,k}$ is $\sigma_u$-sub-Gaussian, i.e.,
\[\E[\exp(\langle a,v^x_{i,k}\rangle)]\le \exp\big(\tfrac{\sigma_x^2\|a\|_2^2}{2}\big),\qquad \E[\exp(\langle b,v^u_{i,k}\rangle)]\le \exp\big(\tfrac{\sigma_u^2\|b\|_2^2}{2}\big)\]
for all $a\in\R^n$ and $b\in\R^m$.
\end{enumerate}
\end{assumption}

Because the trajectories solve a deterministic ODE, each component of $t\mapsto x_i(t)$ is a smooth function.
Therefore, dense sampling enables temporal denoising.
We use the following local-averaging smoother.

\begin{definition}[Local-averaging smoother]\label{def:smoother}
Fix a denoising bandwidth $\BW\in[\Delta t,T]$ and let $M:=\lfloor \BW/\Delta t\rfloor$.
For each $k\in\{0,\dots,L\}$, define
\[
I_k:=\{j\in\{0,\dots,L\}: |j-k|\le M\},
\qquad
\nu_k:=|I_k|.
\]
The denoised samples are
\begin{equation}\label{eq:smoother}
\hat x_{i,k}:=\frac{1}{\nu_k}\sum_{j\in I_k}\tilde x_{i,j},\qquad
\hat u_{i,k}:=\frac{1}{\nu_k}\sum_{j\in I_k}\tilde u_{i,j},
\qquad k=0,\dots,L.
\end{equation}
Thus, boundary nodes are handled by truncated windows automatically. Let $\hat X_k$ and $\hat U_k$ be the block matrices formed from $\{\hat x_{i,k}\}$ and $\{\hat u_{i,k}\}$.
\end{definition}

\subsection{Reconstruction procedure}\label{sec:alg-reconstruction}

\subsubsection{Closed-loop reconstruction: estimating $K^\star(\cdot)$ and $A_c^\star(\cdot)$}\label{sec:alg-closedloop}
We start from the closed-loop relation $U(t)=-K^\star(t)X(t)$.
On the sampling grid, we compute the pointwise least-squares estimator
\begin{equation}\label{eq:Khat-ls}
\hat K_k:=-\hat U_k\hat X_k^\top\big(\hat X_k\hat X_k^\top\big)^{-1},\qquad k=0,1,\dots,L,
\end{equation}
provided that $\hat X_k\hat X_k^\top$ is invertible.
When the cost-recovery stage requires a continuous gain trajectory, we use the piecewise-linear interpolant $\hat K(\cdot)$ satisfying $\hat K(t_k)=\hat K_k$ for all $k$.
This auxiliary interpolation introduces no additional estimation stage; it only supplies a continuous-time coefficient for the Lyapunov equations used below.
As a computable proxy for the state-richness index $c_X$ in \eqref{eq:cX}, we report
\begin{equation}\label{eq:cX-hat}
\hat c_X:=\min_{0\le k\le L}\ \lambda_{\min}(\hat X_k\hat X_k^\top).
\end{equation}

We next estimate the closed-loop dynamics matrix $A_c^\star(\cdot)$ without numerical differentiation.
Fix an even window length $\Delta=\ell\Delta t$ with $\ell\ge 2$.
For $k=0,1,\dots,L-\ell$, define
\[
\hat Y_k:=\hat X_{k+\ell}-\hat X_k,
\qquad
\hat Y_k^{(1/2)}:=\hat X_{k+\ell/2}-\hat X_k,
\qquad
\hat G_k:=\hat X_k\hat X_k^\top.
\]
The one-window increment quotient is introduced only as an intermediate quantity
\begin{equation}\label{eq:Ac-local-est}
\widetilde A_{c,\Delta}(t_k):=\frac{1}{\Delta}\,\hat Y_k\hat X_k^\top\hat G_k^{-1},
\end{equation}
and the estimator used throughout the remainder of the paper is the Richardson-corrected version \cite{richardson1911approx,joyce1971survey}
\begin{equation}\label{eq:Ac-local-est2}
\hat A_c(t_k):=\frac{1}{\Delta}\,\big(4\hat Y_k^{(1/2)}-\hat Y_k\big)\hat X_k^\top\hat G_k^{-1}.
\end{equation}
Here, classical Richardson extrapolation is used purely as a local bias-cancelation device for reconstructing the continuous-time closed-loop dynamics matrix $A_c^\star(\cdot)$ from sampled data. The correction cancels the leading local discretization term while remaining derivative free.
Whenever the cost stage requires a continuous-time closed-loop dynamics matrix, we use the piecewise-linear interpolant through the values in \eqref{eq:Ac-local-est2} on $[0,T-\Delta]$ and extend it to $[T-\Delta,T]$ by the terminal value at $t_{L-\ell}$.
This reconstructed closed-loop dynamics matrix is used both in the dynamics stage and in the cost-recovery stage below.

\subsubsection{Recovery of system matrices and cost weights}\label{sec:alg-AB}
With $\{\hat K_k\}$ and $\hat A_c(\cdot)$ available, we recover the constant pair $(A^\star,B^\star)$ from the discretized counterpart of Theorem~\ref{thm:AB-ident}.
Throughout the sampled-data reconstruction, $w_k$ denotes the weight associated with the uniform grid $t_k=k\Delta t$, namely,
\[
w_0=w_L=\frac{\Delta t}{2},
\qquad
w_k=\Delta t,\quad 1\le k\le L-1.
\]
Hence, $\sum_{k=0}^{L} w_k=T$, the quantities below are weighted time averages that approximate their continuous-time counterparts.
Define
\[
\bar{\hat K}:=\frac{1}{T}\sum_{k=0}^{L} w_k\,\hat K_k,
\qquad
\bar{\hat A}_c:=\frac{1}{T}\sum_{k=0}^{L} w_k\,\hat A_c(t_k),
\]
and let $\Delta\hat K_k:=\hat K_k-\bar{\hat K}$ and $\Delta\hat A_c(t_k):=\hat A_c(t_k)-\bar{\hat A}_c$.
We then form the empirical gain-variation Gramian and its associated cross term
\begin{equation}\label{eq:GK-hat}
\hat{\mathcal G}_K:=\sum_{k=0}^{L} w_k\,\Delta\hat K_k\,\Delta\hat K_k^\top,
\qquad
\hat C_{AK}:=\sum_{k=0}^{L} w_k\,\Delta\hat A_c(t_k)\,\Delta\hat K_k^\top,
\qquad
\hat c_{AB}:=\lambda_{\min}(\hat{\mathcal G}_K),
\end{equation}
and, whenever $\hat c_{AB}$ is bounded away from zero, compute
\begin{equation}\label{eq:AB-hat}
\hat B:=-\hat C_{AK}\hat{\mathcal G}_K^{-1},
\qquad
\hat A:=\bar{\hat A}_c+\hat B\bar{\hat K}.
\end{equation}
The quantity $\hat c_{AB}$ is the empirical diagnostic for this stage.
Its continuous-time counterpart is $c_{AB}$ from \eqref{eq:GK}, and the discretized counterpart introduced in Section~\ref{sec:noise} is $c_{AB,L}$.
The same three-level convention is used for the cost index: $c_{QRH}$ in continuous time, $c_{QRH,L}$ on the sampling grid with the true model, and $\hat c_{QRH}$ for the empirical matrix assembled from $(\hat A_c,\hat B,\hat K)$.
We next recover the cost weighting matrices.
Under Assumption~\ref{ass:terminal-prop}, the terminal weight is parameterized by $H=\alpha Q$. Hence, the free cost variables are $(Q,R)$.
Using the same interpolants $\hat K(\cdot)$ and $\hat A_c(\cdot)$ as above, we impose the sampled stationarity relation
\begin{equation}\label{eq:stationarity-hat}
\hat B^\top P(t)\approx R\hat K(t),
\end{equation}
where $P(\cdot)$ solves
\begin{equation}\label{eq:P-hat}
-\dot P=\hat A_c^\top P+P\hat A_c+Q+\hat K^\top R\hat K,
\qquad P(T)=\alpha Q.
\end{equation}
Once $\hat A_c(\cdot)$ and $\hat K(\cdot)$ are fixed, \eqref{eq:P-hat} is linear in the running state weight, the running input weight, and the terminal matrix.
Accordingly, we represent $P(\cdot)$ as the superposition of basis responses.
Let $\{E_j^Q\}_{j=1}^{d_Q}$ and $\{E_j^R\}_{j=1}^{d_R}$ be Frobenius-orthonormal bases of $\mathbb S^n$ and $\mathbb S^m$, respectively.
Parameterize
\begin{equation}\label{eq:QRH-basis}
Q=\sum_{j=1}^{d_Q} q_j E_j^Q,
\qquad
R=\sum_{j=1}^{d_R} r_j E_j^R,
\qquad
H=\alpha Q,
\end{equation}
where $d_Q:=\frac{n(n+1)}{2}$ and $d_R:=\frac{m(m+1)}{2}$ denote the dimensions of the symmetric matrix spaces
$\mathbb{S}^n$ and $\mathbb{S}^m$, respectively, and $q:=(q_1,\dots,q_{d_Q})^\top\in\R^{d_Q}$ and $r:=(r_1,\dots,r_{d_R})^\top\in\R^{d_R}$ are the coefficient vectors.
We then collect the free coefficients into $\theta:=(q^\top,r^\top)^\top\in\R^{d_\theta}$ with $d_{\theta}=d_Q+d_R$.
Because the bases are Frobenius-orthonormal, every perturbation pair $(\Delta Q,\Delta R)$ represented by $\delta\theta$ satisfies
$\|\delta\theta\|_2=\|(\Delta Q,\Delta R)\|_F$.
For each $Q$-basis element, $P_j^Q$ is the response to the running term $E_j^Q$ with zero terminal value, while $P_j^H$ is the response to the terminal condition $E_j^Q$ with zero running term.
For each $R$-basis element, $P_j^R$ is the response to the running term $\hat K^\top E_j^R\hat K$ with zero terminal value.
Namely, these basis responses satisfy the terminal-value Lyapunov equations
\begin{equation}\label{eq:PjQ}
-\dot P_j^Q=\hat A_c^\top P_j^Q+P_j^Q\hat A_c+E_j^Q,
\qquad P_j^Q(T)=0,
\end{equation}
\begin{equation}\label{eq:PjR}
-\dot P_j^R=\hat A_c^\top P_j^R+P_j^R\hat A_c+\hat K^\top E_j^R\hat K,
\qquad P_j^R(T)=0,
\end{equation}
and
\begin{equation}\label{eq:PjH}
-\dot P_j^H=\hat A_c^\top P_j^H+P_j^H\hat A_c,
\qquad P_j^H(T)=E_j^Q.
\end{equation}
By the linearity of \eqref{eq:P-hat}, the corresponding solution is
\begin{equation}\label{eq:P-decomp-hat}
P(t)=\sum_{j=1}^{d_Q} q_j\big(P_j^Q(t)+\alpha P_j^H(t)\big)
+\sum_{j=1}^{d_R} r_j P_j^R(t),\qquad t\in[0,T].
\end{equation}
Substituting \eqref{eq:P-decomp-hat} into \eqref{eq:stationarity-hat} gives, for each grid point $t_k$,
\begin{equation}\label{eq:residual-linear}
\hat B^\top P(t_k)-R\hat K(t_k)
=\sum_{j=1}^{d_Q} q_j\,\hat B^\top\!\big(P_j^Q(t_k)+\alpha P_j^H(t_k)\big)
+\sum_{j=1}^{d_R} r_j\,\big(\hat B^\top P_j^R(t_k)-E_j^R\hat K(t_k)\big).
\end{equation}
For each $k$, the residual $\hat B^\top P(t_k)-R\hat K(t_k)$ is an $m\times n$ matrix and is linear in the coefficient vector $\theta=(q^\top,r^\top)^\top\in\mathbb R^{d_\theta}$. We therefore define $\hat{\mathbf M}_k\in\mathbb R^{mn\times d_\theta}$ columnwise as follows: its first $d_Q$ columns are
\[
\vecop\big(\hat B^\top(P_j^Q(t_k)+\alpha P_j^H(t_k))\big),\qquad j=1,\dots,d_Q,
\]
and its last $d_R$ columns are
\[
\vecop\big(\hat B^\top P_j^R(t_k)-E_j^R\hat K(t_k)\big),\qquad j=1,\dots,d_R.
\]
Thus, \eqref{eq:residual-linear} is equivalent to
\[
\vecop\big(\hat B^\top P(t_k)-R\hat K(t_k)\big)=\hat{\mathbf M}_k\theta.
\]
We collect these weighted block rows over the grid using the weights $w_0=w_L=\Delta t/2$ and $w_k=\Delta t$ for $1\le k\le L-1$, obtaining the matrix
\begin{equation}\label{eq:MLhat}
\hat M_L:=\begin{bmatrix}
\sqrt{w_0}\,\hat{\mathbf M}_0\\
\vdots\\
\sqrt{w_L}\,\hat{\mathbf M}_L
\end{bmatrix}\in\R^{(L+1)mn\times d_\theta}.
\end{equation}
Thus, $\hat M_L\theta$ is the weighted stacked stationarity residual, and
\[
\|\hat M_L\theta\|_2^2=\sum_{k=0}^{L} w_k\,\|\hat{\mathbf M}_k\theta\|_2^2.
\]
The corresponding noiseless matrix $M_L$ is defined in Section~\ref{sec:noise} by the same weighted stacking rule with $(A_c^\star,B^\star,K^\star)$ in place of $(\hat A_c,\hat B,\hat K)$.
In the noiseless sampled setting, the true normalized coefficient vector satisfies $M_L\theta^\star=0$.
Under noise and discretization, we estimate $(Q,R)$ by the convex program
\begin{equation}\label{eq:QRH-sdp}
\min_{\theta}\ \|\hat M_L\theta\|_2^2
\quad\text{s.t.}\quad
Q(\theta)\succeq 0,\ \ R(\theta)\succeq \epsilon I,\ \ \mathrm{tr}(R(\theta))=m,
\end{equation}
where $\epsilon>0$ is a small numerical margin. For the perturbation analysis, we assume that $\epsilon$ is chosen below the true coercivity level of the input weight, for example $0<\epsilon<\lambda_{\min}(R^\star)$, so that the true parameter vector remains feasible.
The empirical cost-conditioning index is the restricted smallest singular value
\begin{equation}\label{eq:cQRH-hat}
\hat c_{QRH}:=\inf_{\delta\theta\in\mathcal V_\theta:\ \|\delta\theta\|_2=1}\ \|\hat M_L\delta\theta\|_2,
\end{equation}
where $\mathcal V_\theta:=\{\delta\theta=(\delta q^\top,\delta r^\top)^\top:\ \mathrm{tr}(\sum_{j=1}^{d_R}\delta r_j E_j^R)=0\}$ is the coefficient-space representation of the scaling-removed subspace.

\begin{proposition}[Well-posedness of the sampled cost stage]
\label{prop:QRH-unique}
Suppose $0<\epsilon\le1$.
If $\hat c_{QRH}>0$, equivalently, if the weighted residual matrix
$\hat M_L$ is injective on $\mathcal V_\theta$, then
\eqref{eq:QRH-sdp} admits a unique minimizer.
In the noiseless sampled setting, if $\theta^\star$ is feasible for
\eqref{eq:QRH-sdp} and satisfies
$\hat M_L\theta^\star=0$, then this unique minimizer is
$\theta^\star$.
\end{proposition}

\begin{proof}
Let $\mathcal F$ denote the feasible set of
\eqref{eq:QRH-sdp}, and define
$\Theta_m:=\{\theta:\operatorname{tr}(R(\theta))=m\}$.
Since the bases in \eqref{eq:QRH-basis} span
$\mathbb S^n$ and $\mathbb S^m$, there exists a coefficient vector
$\theta_0$ such that $Q(\theta_0)=0$ and $R(\theta_0)=I_m$.
Because $0<\epsilon\le1$, we have
$Q(\theta_0)\succeq0$,
$R(\theta_0)\succeq\epsilon I_m$, and
$\operatorname{tr}(R(\theta_0))=m$.
Hence, $\theta_0\in\mathcal F$, so $\mathcal F$ is nonempty.

The maps $\theta\mapsto Q(\theta)$ and
$\theta\mapsto R(\theta)$ are linear and continuous.
Since the positive semidefinite cones are closed and convex, and
the trace constraint is affine, $\mathcal F$ is closed and convex.
Every $\theta\in\Theta_m$ can be written as
$\theta=\theta_0+\delta\theta$ with
$\delta\theta\in\mathcal V_\theta$. By the definition of
$\hat c_{QRH}$,
\[
\|\hat M_L\delta\theta\|_2
\ge
\hat c_{QRH}\|\delta\theta\|_2,
\qquad
\delta\theta\in\mathcal V_\theta.
\]
For $f(\theta):=\|\hat M_L\theta\|_2^2$, it follows that
\[
\begin{aligned}
f(\theta_0+\delta\theta)=
\|\hat M_L\theta_0+\hat M_L\delta\theta\|_2^2 \ge
\frac{1}{2}\|\hat M_L\delta\theta\|_2^2
-
\|\hat M_L\theta_0\|_2^2 &\ge
\frac{\hat c_{QRH}^2}{2}\|\delta\theta\|_2^2
-
\|\hat M_L\theta_0\|_2^2.
\end{aligned}
\]
Thus, $\hat c_{QRH}>0$ implies that $f$ is coercive on
$\Theta_m$, and hence on $\mathcal F$. The sublevel set
$\{\theta\in\mathcal F:f(\theta)\le f(\theta_0)\}$ is nonempty,
closed, and bounded, and is therefore compact. Since $f$ is
continuous, it attains its minimum over $\mathcal F$.

For any $\theta_1,\theta_2\in\mathcal F$ and $s\in(0,1)$, we have
$\theta_1-\theta_2\in\mathcal V_\theta$ and
\[
\begin{aligned}
f\bigl((1-s)\theta_1+s\theta_2\bigr)
&=
(1-s)f(\theta_1)+sf(\theta_2)-s(1-s)\|\hat M_L(\theta_1-\theta_2)\|_2^2 \\
&\le
(1-s)f(\theta_1)+sf(\theta_2)-s(1-s)\hat c_{QRH}^2
\|\theta_1-\theta_2\|_2^2.
\end{aligned}
\]
Therefore, $f$ is strictly convex on $\mathcal F$, and the
minimizer is unique.

Finally, in the noiseless sampled setting, suppose that
$\theta^\star$ is feasible and satisfies
$\hat M_L\theta^\star=0$. Then
$f(\theta^\star)=0$, which is the minimum possible objective value.
By the uniqueness, $\theta^\star$ is the unique minimizer of
\eqref{eq:QRH-sdp}.
\end{proof}

\subsection{Summary algorithm}\label{sec:implementation}
The resulting sampled-data reconstruction method is summarized in Algorithm~\ref{alg:fhct-joint}. It follows the identifiability decomposition established in Section~\ref{sec:identifiability}. We refer to it as the conditioning-aware sampled-data reconstruction method (CR-IOC). Its main user-selected tuning parameters are the denoising bandwidth $\BW$ and the increment window $\Delta$. Across all stages, small values of $\hat c_X$, $\hat c_{AB}$, or $\hat c_{QRH}$ should be interpreted as intrinsic ill-conditioning rather than as a mere implementation issue.

\begin{algorithm}[H]
\caption{CR-IOC from sampled closed-loop trajectories}\label{alg:fhct-joint}
\begin{algorithmic}[1]
\Require Sampled measurements $\{(\tilde x_{i,k},\tilde u_{i,k})\}_{i=1,k=0}^{N,L}$, horizon $T$, sampling step $\Delta t$, denoising bandwidth $\BW$, even increment window $\Delta=\ell\Delta t$, terminal proportionality factor $\alpha$, and SDP margin $\epsilon$
\Ensure Estimates $(\hat A,\hat B,\hat Q,\hat R,\hat H)$ and diagnostics $(\hat c_X,\hat c_{AB},\hat c_{QRH})$
\State Apply the temporal denoiser in Definition~\ref{def:smoother} to obtain $(\hat x_{i,k},\hat u_{i,k})$ and the block matrices $(\hat X_k,\hat U_k)$; 
\For{$k=0,1,\dots,L$}
    \State Form $\hat G_k=\hat X_k\hat X_k^\top$ and compute $\hat K_k$ from \eqref{eq:Khat-ls}
\EndFor
\State Compute $\hat c_X$ from \eqref{eq:cX-hat} 
\For{$k=0,1,\dots,L-\ell$}
    \State Compute $\hat A_c(t_k)$ from the Richardson-corrected local increment formula \eqref{eq:Ac-local-est2}
\EndFor
\State Form $\bar{\hat K}$, $\bar{\hat A}_c$, $\hat{\mathcal G}_K$, and $\hat c_{AB}$ from \eqref{eq:GK-hat}, then recover $(\hat A,\hat B)$ from \eqref{eq:AB-hat}
\State Assemble $\hat M_L$ from the weighted block rows defined around \eqref{eq:MLhat}, compute $\hat c_{QRH}$ from \eqref{eq:cQRH-hat}, and solve \eqref{eq:QRH-sdp} for $(\hat Q,\hat R)$
\State Set $\hat H=\alpha\hat Q$
\end{algorithmic}
\end{algorithm}

\section{Perturbation analysis and consistency of the sampled-data reconstruction method}\label{sec:noise}
This section studies perturbation propagation through the sampled-data reconstruction method. The analysis follows Algorithm~\ref{alg:fhct-joint}: the denoiser first controls the grid-point denoising errors, these errors then propagate to $K^\star$ and $A_c^\star$, next to the recovery of $(A^\star,B^\star)$, and finally to the structured cost stage. This staged organization keeps the role of the conditioning indices explicit: $c_X$ controls the algebraic inversions in closed-loop reconstruction, $c_{AB}$ controls separation of the open-loop dynamics, and $c_{QRH}$ controls the final cost recovery problem. The observation-noise model and the temporal denoiser were introduced in Section~\ref{sec:algorithm}; we now add the regularity assumption used throughout the perturbation analysis. Throughout this section, we distinguish the continuous-time indices $(c_X,c_{AB},c_{QRH})$, the discretized indices $(c_{X,L},c_{AB,L},c_{QRH,L})$, and the empirical diagnostics $(\hat c_X,\hat c_{AB},\hat c_{QRH})$ computed by the algorithm.

\subsection{Closed-loop perturbation bounds}\label{sec:noise-closedloop}

\begin{assumption}\label{ass:regularity}
There exist finite constants $M_x,M_u,M_{A_c},L_x,L_u$ such that, for all trajectories $i=\{1,\cdots,N\}$ and all $t,s\in[0,T]$,
\begin{equation}\label{eq:reg-bounds}
\begin{aligned}
\|x_i(t)\|_2 &\le M_x, &\qquad \|u_i(t)\|_2 &\le M_u, &\qquad \|A_c^\star(t)\| &\le M_{A_c},\\
\|x_i(t)-x_i(s)\|_2 &\le L_x|t-s|, &\qquad \|u_i(t)-u_i(s)\|_2 &\le L_u|t-s|.
\end{aligned}
\end{equation}
\end{assumption}

Under Assumptions~\ref{ass:standing} and \ref{ass:terminal-prop}, the finite-horizon Riccati solution $P^\star(\cdot)$ is continuously differentiable on $[0,T]$.
Hence, $K^\star(\cdot)=(R^\star)^{-1}B^{\star\top}P^\star(\cdot)$ is continuously differentiable and $A_c^\star(\cdot)=A^\star-B^\star K^\star(\cdot)$ is bounded and Lipschitz on $[0,T]$.
Therefore, Assumption~\ref{ass:regularity} holds whenever the initial conditions are uniformly bounded, because $\dot x_i(t)=A_c^\star(t)x_i(t)$ is then uniformly bounded and
$\dot u_i(t)=-\dot K^\star(t)x_i(t)-K^\star(t)A_c^\star(t)x_i(t)$ is uniformly bounded as well.

With $I_k$ and $\nu_k$ defined in Definition~\ref{def:smoother}, we define the uniform grid-point denoising errors
\begin{equation}\label{eq:delta-def}
\delta_x:=\max_{i,k}\|\hat x_{i,k}-x_i(t_k)\|_2,
\qquad
\delta_u:=\max_{i,k}\|\hat u_{i,k}-u_i(t_k)\|_2.
\end{equation}
For a confidence level $\delta\in(0,1)$, set
\begin{equation}\label{eq:smoothing-bound}
\rho_x:=L_x \BW+\sigma_x\sqrt{\frac{2n\log\!\big(4nN(L+1)/\delta\big)}{M+1}},
\qquad
\rho_u:=L_u \BW+\sigma_u\sqrt{\frac{2m\log\!\big(4mN(L+1)/\delta\big)}{M+1}}.
\end{equation}

\begin{lemma}\label{lem:smoothing-bound}
Under Assumptions~\ref{ass:noise} and \ref{ass:regularity}, with probability at least $1-\delta$,
\begin{equation}\label{eq:smoothing-bound-u}
\delta_x\le \rho_x,
\qquad
\delta_u\le \rho_u.
\end{equation}
\end{lemma}

\begin{proof}
We prove the state bound; the control bound is similar. Fix $(i,k)$ and write the truncated smoother as
\[
\hat x_{i,k}=\frac{1}{\nu_k}\sum_{j\in I_k}\tilde x_{i,j}.
\]
Then,
\[
\hat x_{i,k}-x_i(t_k)
=\frac{1}{\nu_k}\sum_{j\in I_k}\big(x_i(t_j)-x_i(t_k)\big)
+\frac{1}{\nu_k}\sum_{j\in I_k} v^x_{i,j}.
\]
By the Lipschitz bound in \eqref{eq:reg-bounds}, we have
\[
\Big\|\frac{1}{\nu_k}\sum_{j\in I_k}\big(x_i(t_j)-x_i(t_k)\big)\Big\|_2
\le \frac{L_x}{\nu_k}\sum_{j\in I_k}|t_j-t_k|
\le L_x M\Delta t
\le L_x \BW.
\]
For the noise average, Assumption~\ref{ass:noise}~\ref{ass:noise-subg} and independence imply that each coordinate of
\[
\bar v_{i,k}:=\frac{1}{\nu_k}\sum_{j\in I_k} v^x_{i,j}
\]
is $(\sigma_x/\sqrt{\nu_k})$-sub-Gaussian. Since $\nu_k\ge M+1$ for every $k$, the standard tail bound \cite[Sec.~2.5]{vershynin2018hdp} gives, for any $t>0$,
\[
\Pr\big(\|\bar v_{i,k}\|_2\ge t\big)
\le 2n\exp\!\Big(-\frac{(M+1)t^2}{2\sigma_x^2 n}\Big).
\]
A union bound over $i\in\{1,\dots,N\}$ and $k\in\{0,\dots,L\}$ shows that
\[
\Pr\!\left(\max_{i,k}\|\bar v_{i,k}\|_2\ge t\right)
\le 2nN(L+1)\exp\!\Big(-\frac{(M+1)t^2}{2\sigma_x^2 n}\Big).
\]
Choosing $t=\sigma_x\sqrt{\frac{2n\log(4nN(L+1)/\delta)}{M+1}}$ yields $\delta_x\le \rho_x$ with probability at least $1-\delta/2$. The same argument gives $\delta_u\le \rho_u$ with probability at least $1-\delta/2$, and a union bound proves the claim.
\end{proof}

On the event in Lemma~\ref{lem:smoothing-bound}, the block errors satisfy
\begin{equation}\label{eq:eta-def}
\max_{0\le k\le L}\|\hat X_k-X(t_k)\|\le \eta_x:=\sqrt{N}\,\rho_x,
\qquad
\max_{0\le k\le L}\|\hat U_k-U(t_k)\|\le \eta_u:=\sqrt{N}\,\rho_u.
\end{equation}
The next theorem propagates these block errors through the pointwise gain estimator.
For a fixed node $t_k$, write
\[
X_k:=X(t_k),
\qquad
U_k:=U(t_k),
\qquad
K_k^\star:=K^\star(t_k),
\qquad
C_{K,k}:=4\|K_k^\star\|\|X_k\|+\frac{8\|K_k^\star\|\|X_k\|^2}{c_X}\big(2\|X_k\|+\eta_x\big).
\]

\begin{theorem}\label{thm:K-perturb}
Suppose $\lambda_{\min}(X_kX_k^\top)\ge c_X$ and
\begin{equation}\label{eq:eta-cond}
2\|X_k\|\eta_x+\eta_x^2\le \frac{c_X}{2}.
\end{equation}
If $\|\hat X_k-X_k\|\le \eta_x$, then $\hat X_k\hat X_k^\top$ is invertible and
\begin{equation}\label{eq:K-perturb-bound}
\|\hat K_k-K_k^\star\|
\le \frac{4\|X_k\|}{c_X}\,\|\hat U_k-U_k\|
+\frac{C_{K,k}}{c_X}\,\|\hat X_k-X_k\|.
\end{equation}
\end{theorem}

\begin{proof}
Set $G_k:=X_kX_k^\top$, $\hat G_k:=\hat X_k\hat X_k^\top$,
$E_X:=\hat X_k-X_k$, and $E_U:=\hat U_k-U_k$.
Since $\lambda_{\min}(G_k)\ge c_X>0$, the matrix $G_k$ is
invertible and $\|G_k^{-1}\|\le c_X^{-1}$.
Moreover,
\[
\|\hat G_k-G_k\|\le2\|X_k\|\|E_X\|+\|E_X\|^2\le2\|X_k\|\eta_x+\eta_x^2
\le\frac{c_X}{2}.
\]
Weyl's inequality then gives $\lambda_{\min}(\hat G_k)\ge c_X/2>0$.
Hence, $\hat G_k$ is invertible and $\|\hat G_k^{-1}\|\le2/c_X$.
By \eqref{eq:Khat-ls} and Lemma~\ref{lem:K-ident}, we have
\[
\hat K_k-K_k^\star=-\hat U_k\hat X_k^\top\hat G_k^{-1}+U_kX_k^\top G_k^{-1}.
\]
Adding and subtracting $U_kX_k^\top\hat G_k^{-1}$ yields
\begin{equation}\label{eq:Kdiff-split}
\hat K_k-K_k^\star
=-\bigl(\hat U_k\hat X_k^\top-U_kX_k^\top\bigr)
\hat G_k^{-1}-U_kX_k^\top
\bigl(\hat G_k^{-1}-G_k^{-1}\bigr).
\end{equation}

Next, we bound the first term in \eqref{eq:Kdiff-split}:
\begin{align*}
\|\hat U_k\hat X_k^\top-U_kX_k^\top\|
&=\|(U_k+E_U)(X_k+E_X)^\top-U_kX_k^\top\|\\
&\le \|E_U\|\,\|X_k\|+\|U_k\|\,\|E_X\|+\|E_U\|\,\|E_X\|.
\end{align*}
Using $\|U_k\|\le \|K_k^\star\|\,\|X_k\|$ and $\|\hat G_k^{-1}\|\le 2/c_X$ gives
\begin{equation}\label{eq:first-term}
\big\|\big(\hat U_k\hat X_k^\top-U_kX_k^\top\big)\hat G_k^{-1}\big\|
\le \frac{2}{c_X}\Big(\|X_k\|\,\|E_U\|+\|K_k^\star\|\,\|X_k\|\,\|E_X\|+\|E_U\|\,\|E_X\|\Big).
\end{equation}
For the second term in \eqref{eq:Kdiff-split}, we use the identity
$\hat G_k^{-1}-G_k^{-1}=G_k^{-1}(G_k-\hat G_k)\hat G_k^{-1}$.
Therefore,
\begin{equation}\label{eq:inv-diff}
\|\hat G_k^{-1}-G_k^{-1}\|\le \|G_k^{-1}\|\,\|\hat G_k-G_k\|\,\|\hat G_k^{-1}\|
\le \frac{2}{c_X^2}\big(2\|X_k\|\,\|E_X\|+\|E_X\|^2\big).
\end{equation}
Also, $\|U_kX_k^\top\|\le \|U_k\|\,\|X_k\|\le \|K_k^\star\|\,\|X_k\|^2$.
Combining the preceding bound with \eqref{eq:inv-diff} yields
\begin{equation}\label{eq:second-term}
\big\|U_kX_k^\top(\hat G_k^{-1}-G_k^{-1})\big\|
\le \frac{2\|K_k^\star\|\,\|X_k\|^2}{c_X^2}\big(2\|X_k\|+\|E_X\|\big)\,\|E_X\|.
\end{equation}

Since $c_X\le \|X_kX_k^\top\|\le \|X_k\|^2$, condition \eqref{eq:eta-cond} implies $\eta_x<\|X_k\|$, and therefore $\|E_U\|\,\|E_X\|\le \|X_k\|\,\|E_U\|$. Substituting \eqref{eq:first-term} and \eqref{eq:second-term} into \eqref{eq:Kdiff-split} then yields \eqref{eq:K-perturb-bound}.
\end{proof}

Under Assumptions~\ref{ass:standing}, \ref{ass:terminal-prop}, and \ref{ass:regularity}, the solution $P^\star(\cdot)$ of the Riccati equation \eqref{eq:DRE}  is smooth on $[0,T]$.
Consequently, $K^\star(\cdot)$, $A_c^\star(\cdot)$, $X(\cdot)$, and $f(\cdot):=A_c^\star(\cdot)X(\cdot)$ are smooth as well.
Define
\[
M_K:=\sup_{t\in[0,T]}\|K^\star(t)\|,
\qquad
L_K:=\sup_{t\in[0,T]}\|\dot K^\star(t)\|,
\qquad
L_{A_c}:=\sup_{t\in[0,T]}\|\dot A_c^\star(t)\|,
\qquad
L_{A_c,1}:=\sup_{t\in[0,T]}\|\ddot A_c^\star(t)\|,
\]
and let$M_X:=\sup_{t\in[0,T]}\|X(t)\|$, then
\[
L_{f,1}:=\sup_{t\in[0,T]}\|\ddot f(t)\|
\le M_X\big(L_{A_c,1}+3M_{A_c}L_{A_c}+M_{A_c}^3\big)<\infty.
\]
Indeed, $\dot X=A_c^\star X$, $\ddot X=(\dot A_c^\star+A_c^{\star 2})X$, and $\ddot f=\ddot A_c^\star X+2\dot A_c^\star\dot X+A_c^\star\ddot X$. We then obtain the second-order closed-loop dynamics matrix estimate.
\begin{theorem}\label{thm:Ac-perturb2}
Let $f(t):=A_c^\star(t)X(t)$, $L_{f,1}:=\sup_{t\in[0,T]}\|\ddot f(t)\|<\infty$, and assume $c_X>0$.
Fix an even window length $\Delta=\ell\Delta t$ with $\ell\ge 2$ and an index $k\le L-\ell$.
If
\begin{equation}\label{eq:Ac2-cond}
\max\big\{\|\hat X_k-X(t_k)\|,\ \|\hat X_{k+\ell/2}-X(t_{k+\ell/2})\|,\ \|\hat X_{k+\ell}-X(t_{k+\ell})\|\big\}\le \eta_x,
\qquad
2\|X(t_k)\|\eta_x+\eta_x^2\le \frac{c_X}{2},
\end{equation}
then the estimator $\hat A_c(t_k)$ in \eqref{eq:Ac-local-est2} is well defined and satisfies
\[
\|\hat A_c(t_k)-A_c^\star(t_k)\|\le \rho_{A_c,k}(\Delta,\eta_x),
\]
where $M_f:=\sup_{t\in[0,T]}\|A_c^\star(t)X(t)\|$ and
\begin{equation}\label{eq:Ac-perturb2-bound}
\rho_{A_c,k}(\Delta,\eta_x):=
\frac{L_{f,1}\,\|X(t_k)\|}{12c_X}\,\Delta^2
+\frac{16\|X(t_k)\|}{c_X}\,\frac{\eta_x}{\Delta}
+\frac{4M_f}{c_X}\,\eta_x
+\frac{8M_f\|X(t_k)\|^2}{c_X^2}\,\eta_x
+\frac{16}{c_X}\,\frac{\eta_x^2}{\Delta}
+\frac{4M_f\|X(t_k)\|}{c_X^2}\,\eta_x^2.
\end{equation}
\end{theorem}

\begin{proof}
Define the noiseless Richardson estimator
\[
A_{c,\Delta}^{(2),\rm disc}(t_k)
:=2A_{c,\Delta/2}^{\rm disc}(t_k)-A_{c,\Delta}^{\rm disc}(t_k),
\]
where
\[
A_{c,\Delta}^{\rm disc}(t_k):=\frac{1}{\Delta}\big(X(t_{k+\ell})-X(t_k)\big)X(t_k)^\top G_k^{-1},
\]
and
\[
A_{c,\Delta/2}^{\rm disc}(t_k)
:=\frac{2}{\Delta}\,(X(t_{k+\ell/2})-X(t_k))\,X(t_k)^\top\,G_k^{-1}.
\]

We proceed in two parts. First, we control the local discretization error of the noiseless Richardson estimator. Then, we bound the perturbation induced by denoising.
Let $\bar f_{\Delta}(t_k):=\frac{1}{\Delta}\int_{t_k}^{t_{k+\ell}}f(\tau)\,d\tau$ and $\bar f_{\Delta/2}(t_k):=\frac{2}{\Delta}\int_{t_k}^{t_{k+\ell/2}}f(\tau)\,d\tau$.
Because $\dot f$ is Lipschitz, it is absolutely continuous, implying that $\ddot f(t)$ exists almost everywhere and $\|\ddot f(t)\|\le L_{f,1}$ almost everywhere on $[0,T]$.
For $s\in[0,\Delta]$, we have
\[
f(t_k+s)=f(t_k)+s\dot f(t_k)+\int_0^s (s-u)\,\ddot f(t_k+u)\,du.
\]
A direct calculation therefore gives
\[
2\bar f_{\Delta/2}(t_k)-\bar f_{\Delta}(t_k)-f(t_k)
=\int_0^{\Delta} K_{\Delta}(u)\,\ddot f(t_k+u)\,du,
\]
where
\[
K_{\Delta}(u)=
\begin{cases}
-u+\dfrac{3u^2}{2\Delta}, & 0\le u\le \dfrac{\Delta}{2},\\[2mm]
-\dfrac{(\Delta-u)^2}{2\Delta}, & \dfrac{\Delta}{2}<u\le \Delta.
\end{cases}
\]
Since $K_{\Delta}(u)\le 0$ on $[0,\Delta]$ and
\[
\int_0^{\Delta}|K_{\Delta}(u)|\,du
=\int_0^{\Delta/2}\left(u-\frac{3u^2}{2\Delta}\right)du
+\int_{\Delta/2}^{\Delta}\frac{(\Delta-u)^2}{2\Delta}\,du
=\frac{\Delta^2}{12},
\]
we obtain
\[
\big\|2\bar f_{\Delta/2}(t_k)-\bar f_{\Delta}(t_k)-f(t_k)\big\|
\le \frac{L_{f,1}}{12}\Delta^2.
\]
Since $f(t_k)=A_c^\star(t_k)X(t_k)$ and $X(t_k)X(t_k)^\top G_k^{-1}=I$, we have
\[
A_{c,\Delta}^{(2),\rm disc}(t_k)-A_c^\star(t_k)
=\big(2\bar f_{\Delta/2}(t_k)-\bar f_{\Delta}(t_k)-f(t_k)\big)\,X(t_k)^\top G_k^{-1},
\]
and thus
\[
\|A_{c,\Delta}^{(2),\rm disc}(t_k)-A_c^\star(t_k)\|
\le \frac{L_{f,1}\,\|X(t_k)\|}{12c_X}\,\Delta^2.
\]

Write $\hat X_k=X(t_k)+E_X$, $\hat X_{k+\ell/2}=X(t_{k+\ell/2})+E_{X,1/2}$ and $\hat X_{k+\ell}=X(t_{k+\ell})+E_{X,1}$.
Define the noiseless increments $Y:=X(t_{k+\ell})-X(t_k)$ and $Y_{1/2}:=X(t_{k+\ell/2})-X(t_k)$, and their noisy counterparts $\hat Y_k$ and $\hat Y_k^{(1/2)}$.
Let
\[
W:=4Y_{1/2}-Y,\qquad \hat W:=4\hat Y_{1/2,k}-\hat Y_k.
\]
Then, using $Y_{1/2}=\int_{t_k}^{t_{k+\ell/2}} f(\tau)\,d\tau$ and $Y=\int_{t_k}^{t_{k+\ell}} f(\tau)\,d\tau$, we can rewrite
$W=4Y_{1/2}-Y=3\int_{t_k}^{t_{k+\ell/2}} f(\tau)\,d\tau-\int_{t_{k+\ell/2}}^{t_{k+\ell}} f(\tau)\,d\tau$, and hence $\|W\|\le 2M_f\Delta$.
Moreover, writing $\hat X_{k}=X(t_k)+E_X$, $\hat X_{k+\ell/2}=X(t_{k+\ell/2})+E_{X,1/2}$, and $\hat X_{k+\ell}=X(t_{k+\ell})+E_{X,1}$ gives
$\hat W-W=4E_{X,1/2}-E_{X,1}-3E_X$, then it implies that $\|\hat W-W\|\le 8\eta_x$ under \eqref{eq:Ac2-cond}.
Moreover, \eqref{eq:Ac2-cond} implies $\|\hat G_k-G_k\|\le c_X/2$, hence $\|\hat G_k^{-1}\|\le 2/c_X$ and
\(\|\hat G_k^{-1}-G_k^{-1}\|\le \frac{2}{c_X^2}(2\|X(t_k)\|\eta_x+\eta_x^2)\).

Using $\hat A_c(t_k)=\frac{1}{\Delta}\hat W\hat X_k^\top\hat G_k^{-1}$ and $A_{c,\Delta}^{(2),\rm disc}(t_k)=\frac{1}{\Delta}W X(t_k)^\top G_k^{-1}$, the identity
\[
\hat A_c(t_k)-A_{c,\Delta}^{(2),\rm disc}(t_k)
=\frac{1}{\Delta}\Big[(\hat W\hat X_k^\top-WX(t_k)^\top)\hat G_k^{-1}+WX(t_k)^\top(\hat G_k^{-1}-G_k^{-1})\Big]
\]
and a direct expansion give
\[
\frac{1}{\Delta}\|(\hat W\hat X_k^\top-WX(t_k)^\top)\hat G_k^{-1}\|
\le \frac{2}{c_X}\Big(\frac{\|\hat W-W\|\,\|X(t_k)\|}{\Delta}+\frac{\|W\|\,\|E_X\|}{\Delta}+\frac{\|\hat W-W\|\,\|E_X\|}{\Delta}\Big)
\]
and
\[
\frac{1}{\Delta}\|WX(t_k)^\top(\hat G_k^{-1}-G_k^{-1})\|
\le \frac{\|W\|\,\|X(t_k)\|}{\Delta}\,\|\hat G_k^{-1}-G_k^{-1}\|.
\]
Substituting $\|\hat W-W\|\le 8\eta_x$, $\|W\|\le 2M_f\Delta$, and $\|E_X\|\le\eta_x$ yields
\[
\|\hat A_c(t_k)-A_{c,\Delta}^{(2),\rm disc}(t_k)\|
\le
\frac{16\|X(t_k)\|}{c_X}\,\frac{\eta_x}{\Delta}
+\frac{4M_f}{c_X}\,\eta_x
+\frac{16}{c_X}\,\frac{\eta_x^2}{\Delta}
+\frac{8M_f\|X(t_k)\|^2}{c_X^2}\,\eta_x
+\frac{4M_f\|X(t_k)\|}{c_X^2}\,\eta_x^2.
\]
Finally, combining this perturbation bound with the approximation bound above gives \(\|\hat A_c(t_k)-A_c^\star(t_k)\|\le \rho_{A_c,k}(\Delta,\eta_x)\).
\end{proof}

For the full-interval reconstruction used below, the Richardson bias, terminal extension, and leading noise-amplification terms are $\mathcal O(\Delta^2)$, $\mathcal O(\Delta)$, and $\mathcal O(\eta_x/\Delta)$, respectively.
The dominant balance therefore suggests the practical rule $\Delta\asymp \eta_x^{1/2}$.

We next define the sampling-grid counterpart of $c_{AB}$ using the exact closed-loop coefficients so that the empirical reconstruction can be compared with its noiseless discrete counterpart. Define the discretized quantities
\[
\bar K_L^\star:=\frac{1}{T}\sum_{k=0}^{L} w_k\,K^\star(t_k),
\qquad
\bar A_{c,L}^\star:=\frac{1}{T}\sum_{k=0}^{L} w_k\,A_c^\star(t_k),
\]
\[
\Delta K^\star_L(t_k):=K^\star(t_k)-\bar K_L^\star,
\qquad
\Delta A_{c,L}^\star(t_k):=A_c^\star(t_k)-\bar A_{c,L}^\star,
\]
\[
\mathcal G_{K^\star,L}:=\sum_{k=0}^{L} w_k\,\Delta K^\star_L(t_k)\Delta K^\star_L(t_k)^\top,
\qquad
C_{A_c^\star K^\star,L}:=\sum_{k=0}^{L} w_k\,\Delta A_{c,L}^\star(t_k)\Delta K^\star_L(t_k)^\top,
\]
and let $c_{AB,L}:=\lambda_{\min}(\mathcal G_{K^\star,L})$.

\begin{lemma}\label{lem:cAB-discrete}
Assume $K^\star(\cdot)$ is Lipschitz on $[0,T]$ with constant $L_K$.
Then,
\[
\|\mathcal G_{K^\star,L}-\mathcal G_{K^\star}\|\le 2TM_KL_K\,\Delta t,
\qquad
|c_{AB,L}-c_{AB}|\le 2TM_KL_K\,\Delta t.
\]
In particular, $c_{AB,L}\to c_{AB}$, as $\Delta t\downarrow0$.
\end{lemma}

\begin{proof}
Let
\[
S:=\int_0^T K^\star(t)K^\star(t)^\top\,dt,
\qquad
S_L:=\sum_{k=0}^{L} w_k\,K^\star(t_k)K^\star(t_k)^\top.
\]
Since $\|K^\star(t)\|\le M_K,\forall t\in [0,T]$ and $K^\star(\cdot)$ is $L_K$-Lipschitz, the map $t\mapsto K^\star(t)K^\star(t)^\top$ is Lipschitz with constant $2M_KL_K$.
The trapezoidal-rule error estimate therefore gives
\[
\|S_L-S\|\le TM_KL_K\,\Delta t.
\]
Applying the same estimate to $t\mapsto K^\star(t)$ yields
\[
\|\bar K_L^\star-\bar K^\star\|
=\frac{1}{T}\left\|\sum_{k=0}^{L} w_k\,K^\star(t_k)-\int_0^T K^\star(t)\,dt\right\|
\le \frac{L_K}{2}\,\Delta t.
\]
Because $\|\bar K_L^\star\|\le M_K$ and $\|\bar K^\star\|\le M_K$, we have
\[
\|\bar K_L^\star\bar K_L^{\star\top}-\bar K^\star\bar K^{\star\top}\|
\le \|\bar K_L^\star-\bar K^\star\|\,\|\bar K_L^\star\|
+\|\bar K^\star\|\,\|\bar K_L^\star-\bar K^\star\|
\le M_KL_K\,\Delta t.
\]
Using
\[
\mathcal G_{K^\star}=S-T\bar K^\star\bar K^{\star\top},
\qquad
\mathcal G_{K^\star,L}=S_L-T\bar K_L^\star\bar K_L^{\star\top},
\]
we obtain
\[
\|\mathcal G_{K^\star,L}-\mathcal G_{K^\star}\|
\le \|S_L-S\|+T\|\bar K_L^\star\bar K_L^{\star\top}-\bar K^\star\bar K^{\star\top}\|
\le 2TM_KL_K\,\Delta t.
\]
The eigenvalue bound follows from Weyl's inequality.
\end{proof}

\begin{theorem}[Dynamics-stage perturbation]\label{thm:AB-perturb}
Let $\mu_G:=\|\hat{\mathcal G}_K-\mathcal G_{K^\star,L}\|$.
If $\mu_G<c_{AB,L}$, then the estimator \eqref{eq:AB-hat} is well defined and satisfies
\begin{equation}\label{eq:B-perturb}
\|\hat B-B^\star\|
\le \frac{1}{c_{AB,L}-\mu_G}\,\|\hat C_{AK}-C_{A_c^\star K^\star,L}\|
+\frac{\|C_{A_c^\star K^\star,L}\|}{c_{AB,L}(c_{AB,L}-\mu_G)}\,\mu_G,
\end{equation}
and
\begin{equation}\label{eq:A-perturb}
\|\hat A-A^\star\|
\le \|\bar{\hat A}_c-\bar A_{c,L}^\star\|
+\|\hat B-B^\star\|\,\|\bar K_L^\star\|
+\|\hat B\|\,\|\bar{\hat K}-\bar K_L^\star\|.
\end{equation}
\end{theorem}

\begin{proof}

Because $A_c^\star(t_k)=A^\star-B^\star K^\star(t_k)$ for every $k$, the weighted averages satisfy
$\bar A_{c,L}^\star=A^\star-B^\star\bar K_L^\star$.
Hence, $\Delta A_{c,L}^\star(t_k)=-B^\star\Delta K^\star_L(t_k)$ and therefore
$C_{A_c^\star K^\star,L}=-B^\star\mathcal G_{K^\star,L}$.
Write $\hat G:=\hat{\mathcal G}_K$, $G:=\mathcal G_{K^\star,L}$, $\hat C:=\hat C_{AK}$, and $C:=C_{A_c^\star K^\star,L}$.
Then,
\[
\hat B-B^\star=-(\hat C-C)\hat G^{-1}-C\big(\hat G^{-1}-G^{-1}\big).
\]
By Weyl's inequality,
\[
\lambda_{\min}(\hat G)\ge \lambda_{\min}(G)-\|\hat G-G\|=c_{AB,L}-\mu_G>0,
\]
then we have $\hat G$ is invertible and
\[
\|\hat G^{-1}\|\le \frac{1}{c_{AB,L}-\mu_G},
\qquad
\|G^{-1}\|\le \frac{1}{c_{AB,L}}.
\]
Using the identity
\[
\hat G^{-1}-G^{-1}=G^{-1}(G-\hat G)\hat G^{-1}
\]
yields
\[
\|\hat G^{-1}-G^{-1}\|
\le \frac{\mu_G}{c_{AB,L}(c_{AB,L}-\mu_G)}.
\]
Taking norms in the decomposition of $\hat B-B^\star$ proves \eqref{eq:B-perturb}.
Finally, since $A^\star=\bar A_{c,L}^\star+B^\star\bar K_L^\star$ and $\hat A=\bar{\hat A}_c+\hat B\bar{\hat K}$,
\[
\hat A-A^\star=(\bar{\hat A}_c-\bar A_{c,L}^\star)+(\hat B-B^\star)\bar K_L^\star+\hat B(\bar{\hat K}-\bar K_L^\star).
\]
Applying the triangle inequality gives \eqref{eq:A-perturb}.
\end{proof}
\subsection{Cost recovery and consistency}\label{sec:noise-cost}
We next study the cost stage and introduce the corresponding discretized matrix obtained from the true model on the same sampling grid.
Here, the same trapezoidal weights are used, namely, $w_0=w_L=\Delta t/2$ and $w_k=\Delta t$ for $1\le k\le L-1$ on the grid $t_k=k\Delta t$.
Define $M_L\in\mathbb R^{(L+1)mn\times d_\theta}$ directly by the same columnwise construction and weighted stacking as $\hat M_L$ in Section~\ref{sec:alg-AB}, with $(A_c^\star,B^\star,K^\star)$ in place of $(\hat A_c,\hat B,\hat K)$.
For any structured perturbation $(\Delta Q,\Delta R)\in\mathcal V_\alpha$, let $\delta\theta\in\mathcal V_\theta$ denote its coefficient vector with respect to the bases $\{E_j^Q\}$ and $\{E_j^R\}$, and define the weighted stacked residual
\begin{equation}\label{eq:ML-def}
\mathsf{M}_L(\Delta Q,\Delta R)
:=\begin{bmatrix}
\sqrt{w_0}\,\vecop\big(\mathscr{M}_\alpha(\Delta Q,\Delta R)(t_0)\big)\\
\vdots\\
\sqrt{w_L}\,\vecop\big(\mathscr{M}_\alpha(\Delta Q,\Delta R)(t_L)\big)
\end{bmatrix}
=M_L\delta\theta.
\end{equation}
We define
\begin{equation}\label{eq:cQRH-L}
c_{QRH,L}:=\inf_{\delta\theta\in\mathcal V_\theta:\ \|\delta\theta\|_2=1}\ \|M_L\delta\theta\|_2.
\end{equation}

For each $(\Delta Q,\Delta R)\in\mathcal V_\alpha$ with $\|(\Delta Q,\Delta R)\|_F=1$, let $P_\Delta(\cdot):=P_{\Delta Q,\Delta R,\alpha\Delta Q}(\cdot)$ solve \eqref{eq:P-operator} with $A_c=A_c^\star$ and $K=K^\star$.
Since $\|\Delta Q\|_F\le 1$ and $\|\Delta R\|_F\le 1$, Lemma~\ref{lem:terminal-lyap} and the bound $\|\Phi(\tau,t)\|\le e^{M_{A_c}(\tau-t)}$ imply
\[
\sup_{t\in[0,T]}\|P_\Delta(t)\|_F
\le C_P:=\alpha e^{2M_{A_c}T}+\Gamma_{A_c}(1+M_K^2),
\qquad
\Gamma_{A_c}:=\int_0^T e^{2M_{A_c}s}\,ds.
\]
Moreover, the differential equation \eqref{eq:P-operator} gives
\[
\sup_{t\in[0,T]}\|\dot P_\Delta(t)\|_F
\le C_{\dot P}:=2M_{A_c}C_P+1+M_K^2.
\]
Therefore, the structured residual family is automatically uniformly bounded and uniformly Lipschitz on the unit sphere of $\mathcal V_\alpha$.
In particular,
\[
\sup_{\|(\Delta Q,\Delta R)\|_F=1}\sup_{t\in[0,T]}\big\|\mathscr{M}_\alpha(\Delta Q,\Delta R)(t)\big\|_F
\le M_F:=\|B^\star\|C_P+M_K,
\]
\[
\sup_{\|(\Delta Q,\Delta R)\|_F=1}\mathrm{Lip}\big(\mathscr{M}_\alpha(\Delta Q,\Delta R)\big)
\le L_F:=\|B^\star\|C_{\dot P}+L_K.
\]
The constants $M_F$ and $L_F$ therefore depend only on $(A^\star,B^\star,K^\star,\alpha,T)$ and are finite under the standing finite-horizon assumptions.

\begin{lemma}\label{lem:cQRH-discrete}
The following estimate holds
\begin{equation}\label{eq:cQRH-discrete-bound}
|c_{QRH,L}^2-c_{QRH}^2|\le T M_F L_F\,\Delta t.
\end{equation}
If $c_{QRH}>0$, then
\[
|c_{QRH,L}-c_{QRH}|\le \frac{T M_F L_F}{c_{QRH}}\,\Delta t.
\]
In particular, $c_{QRH,L}\to c_{QRH}$, as $\Delta t\downarrow0$.
\end{lemma}

\begin{proof}
Define, for each unit-norm $(\Delta Q,\Delta R)\in\mathcal{V}_\alpha$,
\[
F_\Delta(t):=\mathscr{M}_\alpha(\Delta Q,\Delta R)(t)\in\R^{m\times n},
\qquad
f_\Delta(t):=\|F_\Delta(t)\|_F^2.
\]
Since $\sup_{t\in[0,T]}\|F_\Delta(t)\|_F\le M_F$, we have, for any $s,t\in[0,T]$,
\begin{align*}
|f_\Delta(t)-f_\Delta(s)|
&=\big|\|F_\Delta(t)\|_F^2-\|F_\Delta(s)\|_F^2\big|\\
&\le (\|F_\Delta(t)\|_F+\|F_\Delta(s)\|_F)\,\|F_\Delta(t)-F_\Delta(s)\|_F\\
&\le 2M_F L_F |t-s|.
\end{align*}
Hence, $f_\Delta(\cdot)$ is uniformly Lipschitz with constant $L_f:=2M_F L_F$.
For each interval $[t_\ell,t_{\ell+1}]$, the trapezoidal quadrature error satisfies
\begin{align*}
&\left|\int_{t_\ell}^{t_{\ell+1}} f_\Delta(t)\,dt - \frac{\Delta t}{2}\big(f_\Delta(t_\ell)+f_\Delta(t_{\ell+1})\big)\right|\\
&\qquad\le \int_{t_\ell}^{t_{\ell+1}} \left|f_\Delta(t)-\frac{f_\Delta(t_\ell)+f_\Delta(t_{\ell+1})}{2}\right|dt
\le \frac{L_f}{2}(\Delta t)^2
= M_F L_F (\Delta t)^2.
\end{align*}
Summing over $\ell=0,\dots,L-1$ yields
\begin{equation}\label{eq:norm-approx}
\big|\|F_\Delta\|_{L^2([0,T])}^2-\|\mathsf{M}_L(\Delta Q,\Delta R)\|_2^2\big|
\le T M_F L_F\,\Delta t.
\end{equation}
Since the unit spheres in $\mathcal V_\alpha$ and
$\mathcal V_\theta$ are compact and the corresponding objective
functions are continuous, the infima defining $c_{QRH}$ and
$c_{QRH,L}$ are attained.
Let $(\Delta Q^\dagger,\Delta R^\dagger)\in\mathcal{V}_\alpha$ with $\|(\Delta Q^\dagger,\Delta R^\dagger)\|_F=1$ attain the infimum defining $c_{QRH}$.
Then, \eqref{eq:norm-approx} gives
\[
c_{QRH,L}^2\le c_{QRH}^2+T M_F L_F\,\Delta t.
\]
Likewise, if $(\Delta Q_L^\dagger,\Delta R_L^\dagger)$ attains the infimum defining $c_{QRH,L}$, then
\[
c_{QRH}^2\le c_{QRH,L}^2+T M_F L_F\,\Delta t.
\]
Combining the last two displays yields \eqref{eq:cQRH-discrete-bound}.
If $c_{QRH}>0$, then
\[
|c_{QRH,L}-c_{QRH}|=
\frac{|c_{QRH,L}^2-c_{QRH}^2|}{c_{QRH,L}+c_{QRH}}
\le \frac{|c_{QRH,L}^2-c_{QRH}^2|}{c_{QRH}}
\le \frac{T M_F L_F}{c_{QRH}}\,\Delta t.
\]
This completes the proof.
\end{proof}
The true normalized parameter vector $\theta^\star$ then satisfies
\begin{equation}\label{eq:MLtheta0}
M_L\theta^\star=0.
\end{equation}

\begin{theorem}[Cost-stage perturbation bound]\label{thm:QRH-perturb}
Assume $c_{QRH,L}>0$ and define
\[
\delta_M:=\|\hat M_L-M_L\|.
\]
If $\delta_M<c_{QRH,L}$, and $\theta^\star$ is feasible for \eqref{eq:QRH-sdp}, then \eqref{eq:QRH-sdp} admits a unique minimizer $\hat\theta$, and
\begin{equation}\label{eq:theta-bound}
\|\hat\theta-\theta^\star\|_2
\le \frac{2}{c_{QRH,L}-\delta_M}\,\|\hat M_L\theta^\star\|_2
=\frac{2}{c_{QRH,L}-\delta_M}\,\|(\hat M_L-M_L)\theta^\star\|_2.
\end{equation}
\end{theorem}

\begin{proof}
By definition, the restriction of $M_L$ to $\mathcal V_\theta$ has smallest singular value $c_{QRH,L}$.
Weyl's inequality and the perturbation assumption imply that the restriction of $\hat M_L$ to $\mathcal V_\theta$ has smallest singular value at least $c_{QRH,L}-\delta_M$.
Hence, $\hat M_L$ is injective on $\mathcal V_\theta$, and Proposition~\ref{prop:QRH-unique} implies that the minimizer of \eqref{eq:QRH-sdp} is unique.

Since $\hat\theta$ minimizes the objective over the feasible set and $\theta^\star$ is feasible, it holds
\[
\|\hat M_L\hat\theta\|_2\le \|\hat M_L\theta^\star\|_2.
\]
Let $\delta\theta:=\hat\theta-\theta^\star$.
Because both $\hat\theta$ and $\theta^\star$ satisfy $\mathrm{tr}(R)=m$, one has $\delta\theta\in\mathcal V_\theta$.
The triangle inequality gives
\[
\|\hat M_L\delta\theta\|_2
\le \|\hat M_L\hat\theta\|_2+\|\hat M_L\theta^\star\|_2
\le 2\|\hat M_L\theta^\star\|_2.
\]
The restricted singular value bound then yields
\[
(c_{QRH,L}-\delta_M)\,\|\delta\theta\|_2\le \|\hat M_L\delta\theta\|_2.
\]
Combining the last two displays proves the first inequality in \eqref{eq:theta-bound}.
Finally, \eqref{eq:MLtheta0} implies $\hat M_L\theta^\star=(\hat M_L-M_L)\theta^\star$.
\end{proof}
For the cost-stage operator perturbation, define
\[
\Delta_{A_c}:=\sup_{t\in[0,T]}\|\hat A_c(t)-A_c^\star(t)\|,
\qquad
\Delta_B:=\|\hat B-B^\star\|,
\qquad
\Delta_K:=\sup_{t\in[0,T]}\|\hat K(t)-K^\star(t)\|,
\]
\[
M_K=\sup_{t\in[0,T]}\|K^\star(t)\|,
\qquad
\hat M_K:=\sup_{t\in[0,T]}\|\hat K(t)\|,
\qquad
\hat M_{A_c}:=\sup_{t\in[0,T]}\|\hat A_c(t)\|,
\]
\[
\bar M_A:=\max\{M_{A_c},\hat M_{A_c}\},
\qquad
\Gamma_A:=\int_0^T e^{2\bar M_A s}\,ds
=
\begin{cases}
T, & \bar M_A=0,\\[1mm]
\dfrac{e^{2\bar M_A T}-1}{2\bar M_A}, & \bar M_A>0,
\end{cases}
\]
\[
\beta_Q:=(\Gamma_A+\alpha e^{2\bar M_A T})\,\Delta_B
+2\|B^\star\|\Gamma_A(\Gamma_A+\alpha e^{2\bar M_A T})\,\Delta_{A_c},
\]
\[
\beta_R:=\hat M_K^2\Gamma_A\,\Delta_B
+\|B^\star\|\Big(2\Gamma_A^2\hat M_K^2\,\Delta_{A_c}
+\Gamma_A(M_K+\hat M_K)\,\Delta_K\Big)
+\Delta_K.
\]

\begin{lemma}\label{lem:ML-lipschitz}
Under Assumption~\ref{ass:regularity}, assume in addition that $\hat M_{A_c}<\infty$. Then,
\begin{equation}\label{eq:ML-lip}
\|\hat M_L-M_L\|
\le \sqrt{T\big(d_Q\beta_Q^2+d_R\beta_R^2\big)}
+\mathrm{err}_{\mathrm{ODE}}
\end{equation}
Here, $\mathrm{err}_{\mathrm{ODE}}$ denotes the numerical error incurred when solving the terminal-value Lyapunov equations that define the columns of $M_L$ and $\hat M_L$.
\end{lemma}

\begin{proof}
We estimate the sampled columns of $\hat M_L-M_L$ in Frobenius norm.
Because the bases $\{E_j^Q\}$ and $\{E_j^R\}$ are Frobenius-orthonormal, every basis element has Frobenius norm one.
Let $\Phi$ and $\hat\Phi$ denote the state-transition matrices associated with $A_c^\star(\cdot)$ and $\hat A_c(\cdot)$, respectively.
By Assumption~\ref{ass:regularity} and the definition of $\bar M_A$, it holds
\[
\|\Phi(\tau,t)\|\le e^{\bar M_A(\tau-t)},
\qquad
\|\hat\Phi(\tau,t)\|\le e^{\bar M_A(\tau-t)},
\qquad 0\le t\le \tau\le T.
\]
Lemma~\ref{lem:terminal-lyap} therefore gives the uniform bounds
\[
\sup_t\|P_j^Q(t)\|_F,\ \sup_t\|\hat P_j^Q(t)\|_F\le \Gamma_A,
\qquad
\sup_t\|P_j^H(t)\|_F,\ \sup_t\|\hat P_j^H(t)\|_F\le e^{2\bar M_A T},
\]
and
\[
\sup_t\|P_j^R(t)\|_F\le M_K^2\Gamma_A,
\qquad
\sup_t\|\hat P_j^R(t)\|_F\le \hat M_K^2\Gamma_A.
\]

For a $Q$-basis column, let $\Delta P_j^Q:=\hat P_j^Q-P_j^Q$.
Subtracting the equations defining $\hat P_j^Q$ and $P_j^Q$ shows that $\Delta P_j^Q$ solves a terminal-value Lyapunov equation with zero terminal value and forcing
\[
(\hat A_c-A_c^\star)^\top\hat P_j^Q+\hat P_j^Q(\hat A_c-A_c^\star).
\]
Using Lemma~\ref{lem:terminal-lyap} again, together with the bound $\sup_t\|\hat P_j^Q(t)\|_F\le \Gamma_A$, yields
\[
\sup_t\|\Delta P_j^Q(t)\|_F\le 2\Gamma_A^2\,\Delta_{A_c}.
\]
The same argument for the terminal-response equation \eqref{eq:PjH} gives
\[
\sup_t\|\hat P_j^H(t)-P_j^H(t)\|_F
\le 2\Gamma_A e^{2\bar M_A T}\,\Delta_{A_c}.
\]
Consequently, for each $j=1,\dots,d_Q$ and every grid point $t_k$, we have
\begin{align*}
&\big\|\hat B^\top\big(\hat P_j^Q(t_k)+\alpha \hat P_j^H(t_k)\big)-B^{\star\top}\big(P_j^Q(t_k)+\alpha P_j^H(t_k)\big)\big\|_F\\
&\qquad\le \Delta_B\,\big(\Gamma_A+\alpha e^{2\bar M_A T}\big)
+2\|B^\star\|\Gamma_A\big(\Gamma_A+\alpha e^{2\bar M_A T}\big)\,\Delta_{A_c}
=\beta_Q.
\end{align*}
Hence, each weighted $Q$-column of $\hat M_L-M_L$ has Euclidean norm at most $\sqrt{T}\,\beta_Q$.

For an $R$-basis column, let $\Delta P_j^R:=\hat P_j^R-P_j^R$.
Subtracting the two terminal-value Lyapunov equations gives a forcing term equal to
\[
(\hat A_c-A_c^\star)^\top\hat P_j^R+\hat P_j^R(\hat A_c-A_c^\star)+\hat K^\top E_j^R\hat K-K^{\star\top} E_j^R K^\star.
\]
Because $\|E_j^R\|_F=1$, it holds
\[
\|\hat K^\top E_j^R\hat K-K^{\star\top} E_j^R K^\star\|_F
\le (\hat M_K+M_K)\,\Delta_K.
\]
Applying Lemma~\ref{lem:terminal-lyap} and using $\sup_t\|\hat P_j^R(t)\|_F\le \hat M_K^2\Gamma_A$ give
\[
\sup_t\|\Delta P_j^R(t)\|_F
\le 2\Gamma_A^2\hat M_K^2\,\Delta_{A_c}
+\Gamma_A(M_K+\hat M_K)\,\Delta_K.
\]
Therefore, for each $j=1,\dots,d_R$ and every grid point $t_k$, we have
\begin{align*}
&\big\|\hat B^\top\hat P_j^R(t_k)-E_j^R\hat K(t_k)-\big(B^{\star\top} P_j^R(t_k)-E_j^R K^\star(t_k)\big)\big\|_F\\
&\qquad\le \hat M_K^2\Gamma_A\,\Delta_B
+\|B^\star\|\Big(2\Gamma_A^2\hat M_K^2\,\Delta_{A_c}
+\Gamma_A(M_K+\hat M_K)\,\Delta_K\Big)
+\Delta_K
=\beta_R.
\end{align*}
Hence, each weighted $R$-column of $\hat M_L-M_L$ has Euclidean norm at most $\sqrt{T}\,\beta_R$.

The Frobenius norm of the full matrix $\hat M_L-M_L$ is therefore bounded by
\[
\|\hat M_L-M_L\|_F
\le \sqrt{T\big(d_Q\beta_Q^2+d_R\beta_R^2\big)}.
\]
Since the operator norm is dominated by the Frobenius norm, this proves \eqref{eq:ML-lip} up to the numerical error incurred by the terminal-value ODE solves.
Adding that contribution yields the stated bound.
\end{proof}
Recall the finite constants
\[
L_K:=\sup_{t\in[0,T]}\|\dot K^\star(t)\|,
\qquad
L_{A_c}:=\sup_{t\in[0,T]}\|\dot A_c^\star(t)\|,
\qquad
L_{A_c,1}:=\sup_{t\in[0,T]}\|\ddot A_c^\star(t)\|
\]
with $L_{A_c}\le \|B^\star\|L_K$.
Moreover, by \eqref{eq:DRE}, one may take
\[
L_K\le \|(R^\star)^{-1}B^{\star\top}\|\sup_{t\in[0,T]}\|A^{\star\top} P^\star(t)+P^\star(t)A^\star-P^\star(t)B^\star (R^\star)^{-1}B^{\star\top} P^\star(t)+Q^\star\|.
\]
Set
\[
\widetilde L_f:=M_X^2\big(L_{A_c,1}+3M_{A_c}L_{A_c}+M_{A_c}^3\big),
\]
and define
\begin{align*}
\varepsilon_K
&:= \frac{4M_X}{c_X}\,\eta_u
+\frac{1}{c_X}\left(4M_KM_X+\frac{8M_KM_X^2}{c_X}(2M_X+\eta_x)\right)\eta_x,\\
\varepsilon_{A_c}
&:= \frac{\widetilde L_f}{12c_X}\,\Delta^2
+\frac{16M_X}{c_X}\,\frac{\eta_x}{\Delta}
+\frac{4M_f}{c_X}\,\eta_x
+\frac{8M_fM_X^2}{c_X^2}\,\eta_x
+\frac{16}{c_X}\,\frac{\eta_x^2}{\Delta}
+\frac{4M_fM_X}{c_X^2}\,\eta_x^2
+L_{A_c}\Delta,\\
\overline\Delta_K&:=\varepsilon_K+\frac{L_K}{2}\,\Delta t,\\
\overline\mu_G&:=T\big(8M_K\overline\Delta_K+4\overline\Delta_K^2\big),\\
\overline\mu_C&:=T\big(4M_{A_c}\overline\Delta_K+4M_K\varepsilon_{A_c}+4\varepsilon_{A_c}\overline\Delta_K\big),\\
\overline\Delta_B&:=\frac{\overline\mu_C}{c_{AB,L}-\overline\mu_G}
+\frac{\|C_{A_c^\star K^\star,L}\|}{c_{AB,L}(c_{AB,L}-\overline\mu_G)}\,\overline\mu_G.
\end{align*}
Set $\overline\Gamma_A:=\int_0^T e^{2(M_{A_c}+\varepsilon_{A_c})s}\,ds$, and let $\overline\beta_Q$ and $\overline\beta_R$ denote the expressions defining $\beta_Q$ and $\beta_R$ above after replacing
$(\Delta_B,\Delta_{A_c},\Delta_K,\hat M_K,\bar M_A,\Gamma_A)$ by
$(\overline\Delta_B,\varepsilon_{A_c},\overline\Delta_K,M_K+\overline\Delta_K,M_{A_c}+\varepsilon_{A_c},\overline\Gamma_A)$.
Finally, set $\overline\mu_M:=\sqrt{T(d_Q\overline\beta_Q^2+d_R\overline\beta_R^2)}+\mathrm{err}_{\mathrm{ODE}}$.

\begin{theorem}[Error propagation]\label{thm:e2e}
Suppose Assumptions~\ref{ass:standing}, \ref{ass:terminal-prop}, \ref{ass:noise},
and~\ref{ass:regularity} hold, the true cost is normalized by $\mathrm{tr}(R^\star)=m$, and
choose $\epsilon<\lambda_{\min}(R^\star)$
in~\eqref{eq:QRH-sdp}.
Assume further that
\begin{equation}\label{eq:e2e-cond}
  2M_X\eta_x+\eta_x^2\le \frac{c_X}{2},
  \qquad
  \overline\mu_G<c_{AB,L},
  \qquad
  \overline\mu_M<c_{QRH,L}.
\end{equation}
Then, with a probability of at least $1-\delta$, we have
\begin{equation}\label{eq:e2e-bounds}
\begin{aligned}
  & \max_k \|\hat K_k - K^\star(t_k)\|
    \le \varepsilon_K, \qquad
  \Delta_K \le \overline\Delta_K, \qquad
  \Delta_{A_c} \le \varepsilon_{A_c},\\[4pt]
  & \|\hat B-B^\star\|\le\overline\Delta_B,\\
  &\|\hat A-A^\star\|
    \le \varepsilon_{A_c}+M_K\overline\Delta_B
       +(\|B^\star\|+\overline\Delta_B)\overline\Delta_K,\\
   &\|\hat\theta-\theta^\star\|_2
   \le \frac{2\|\theta^\star\|_2}
              {c_{QRH,L}-\overline\mu_M}\,\overline\mu_M .
\end{aligned}
\end{equation}
Moreover, if $c_{AB}>0$, $c_{QRH}>0$, and
$\eta_x,\eta_u\to0$, $\Delta\downarrow0$,
$\eta_x/\Delta\to0$, $\mathrm{err}_{\mathrm{ODE}}\to0$,
then \eqref{eq:e2e-cond} is eventually satisfied, and
$(\hat A,\hat B,\hat Q,\hat R,\hat H)$ is a consistent estimator
for $(A^\star,B^\star,Q^\star,R^\star,H^\star)$.
\end{theorem}
\begin{proof}
By Lemma~\ref{lem:smoothing-bound} and \eqref{eq:eta-def}, with probability at least $1-\delta$ the denoised blocks satisfy
\[
\|\hat X_k-X(t_k)\|\le \eta_x,
\qquad
\|\hat U_k-U(t_k)\|\le \eta_u,
\]
uniformly in $k$. We work under this event throughout.

Applying Theorem~\ref{thm:K-perturb} with $\|X(t_k)\|\le M_X$ and $\|K^\star(t_k)\|\le M_K$ and using the first condition in~\eqref{eq:e2e-cond} give $\max_k\|\hat K_k-K^\star(t_k)\|\le\varepsilon_K$. Because $\hat K(\cdot)$ is the piecewise-linear interpolant through $\{\hat K_k\}_{k=0}^{L}$ while $K^\star(\cdot)$ is $L_K$-Lipschitz, we have
\[
\Delta_K
\le \max_k\|\hat K_k-K^\star(t_k)\|+\frac{L_K}{2}\,\Delta t
\le \overline\Delta_K.
\]
For the closed-loop dynamics matrix, Theorem~\ref{thm:Ac-perturb2} gives the grid-point estimate
\[
\max_{0\le k\le L-\ell}
  \|\hat A_c(t_k)-A_c^\star(t_k)\|
\le \varepsilon_{A_c}-L_{A_c}\Delta.
\]
On each interval $[t_k,t_{k+1}]\subset[0,T-\Delta]$, the piecewise-linear interpolation error of the true $A_c^\star(\cdot)$ is at most $L_{A_c}\Delta t/2$. On the terminal interval $[T-\Delta,T]$, the algorithm extends $\hat A_c(\cdot)$ by the value at $t_{L-\ell}$, and the Lipschitz continuity of $A_c^\star(\cdot)$ gives an extension error at most $L_{A_c}\Delta$. Hence, $\Delta_{A_c}\le\varepsilon_{A_c}$. This establishes the first line of~\eqref{eq:e2e-bounds}.             

Let $e_k:=\hat K_k-K^\star(t_k)$ and $\bar e:=T^{-1}\sum_{k=0}^{L}w_k e_k$.
Then, $\Delta\hat K_k=\Delta K^\star_L(t_k)+(e_k-\bar e)$ and
\[
\|e_k-\bar e\|\le 2\Delta_K\le 2\overline\Delta_K,
\qquad
\|\Delta K^\star_L(t_k)\|\le 2M_K.
\]
Therefore,
\[
\|\hat{\mathcal G}_K-\mathcal G_{K^\star,L}\|
\le \sum_{k=0}^{L}w_k
  \Big(2\|\Delta K^\star_L(t_k)\|\,\|e_k-\bar e\|
       +\|e_k-\bar e\|^2\Big)
\le \overline\mu_G.
\]
Similarly, if $f_k:=\hat A_c(t_k)-A_c^\star(t_k)$ and $\bar f:=T^{-1}\sum_{k=0}^{L}w_k f_k$, then
\[
\|f_k-\bar f\|\le 2\Delta_{A_c}\le 2\varepsilon_{A_c},
\qquad
\|\Delta A_{c,L}^\star(t_k)\|\le 2M_{A_c}.
\]
Hence,
\[
\|\hat C_{AK}-C_{A_c^\star K^\star,L}\|
\le T\big(4M_{A_c}\overline\Delta_K
         +4M_K\varepsilon_{A_c}
         +4\varepsilon_{A_c}\overline\Delta_K\big)
=\overline\mu_C.
\]
Theorem~\ref{thm:AB-perturb} and the second condition in~\eqref{eq:e2e-cond} give $\|\hat B-B^\star\|\le\overline\Delta_B$.                
                                        
Using
$\|\bar{\hat A}_c-\bar A_{c,L}^\star\|\le\varepsilon_{A_c}$,
$\|\bar K_L^\star\|\le M_K$, and
$\|\bar{\hat K}-\bar K_L^\star\|\le\overline\Delta_K$
in~\eqref{eq:A-perturb} gives the estimate for $\|\hat A-A^\star\|$.
Moreover, $\Delta_{A_c}\le\varepsilon_{A_c}$ and $\Delta_K\le\overline\Delta_K$ imply
$\bar M_A\le M_{A_c}+\varepsilon_{A_c}$, $\Gamma_A\le\overline\Gamma_A$, and $\hat M_K\le M_K+\overline\Delta_K$.
Together with $\Delta_B\le\overline\Delta_B$, the definitions above give $\beta_Q\le\overline\beta_Q$ and $\beta_R\le\overline\beta_R$.
Hence, Lemma~\ref{lem:ML-lipschitz} yields $\delta_M\le\overline\mu_M$. Applying Theorem~\ref{thm:QRH-perturb} and the third condition in~\eqref{eq:e2e-cond} gives the last line of~\eqref{eq:e2e-bounds}.

It remains to prove the consistency. Assume now that $c_{AB}>0$, $c_{QRH}>0$,
$\eta_x,\eta_u\to0$, $\Delta\downarrow0$, $\eta_x/\Delta\to0$, and $\mathrm{err}_{\mathrm{ODE}}\to0$. Because Assumption~\ref{ass:richIC} implies $c_X>0$
and $\eta_x\to0$, the first condition in~\eqref{eq:e2e-cond}
holds eventually. The definitions of $\varepsilon_K$, $\varepsilon_{A_c}$, and $\overline\Delta_K$ then give
\[
\varepsilon_K\to0,
\qquad
\varepsilon_{A_c}\to0,
\qquad
\overline\Delta_K\to0,
\qquad
\overline\mu_G\to0,
\qquad
\overline\mu_C\to0.
\]
Since $\Delta=\ell\Delta t$ with $\ell\ge2$, the condition $\Delta\downarrow0$ implies $\Delta t\downarrow0$. Lemma~\ref{lem:cAB-discrete} therefore yields $c_{AB,L}\to c_{AB}$, and $c_{AB,L}\ge c_{AB}/2>0$ eventually. Hence, the second condition in~\eqref{eq:e2e-cond} holds eventually. Since $C_{A_c^\star K^\star,L}=-B^\star\mathcal G_{K^\star,L}$ and $\|\mathcal G_{K^\star,L}\|\le4TM_K^2$, the sequence $\|C_{A_c^\star K^\star,L}\|$ is bounded. Thus, $\overline\Delta_B\to0$, and the bounds above imply $\|\hat B-B^\star\|\to0$ and $\|\hat A-A^\star\|\to0$.
Moreover, $\overline\Gamma_A$ remains bounded. Because $\overline\Delta_K\to0$, $\varepsilon_{A_c}\to0$, and $\overline\Delta_B\to0$, the definitions of $\overline\beta_Q$ and $\overline\beta_R$ give $\overline\beta_Q\to0$ and $\overline\beta_R\to0$. Therefore, $\overline\mu_M\to0$. Lemma~\ref{lem:cQRH-discrete} yields $c_{QRH,L}\to c_{QRH}$, $c_{QRH,L}\ge c_{QRH}/2>0$ eventually,
and hence the third condition in~\eqref{eq:e2e-cond} holds eventually. The last line of~\eqref{eq:e2e-bounds} then gives $\|\hat\theta-\theta^\star\|_2\to0$, which is equivalent to
$(\hat Q,\hat R,\hat H)\to(Q^\star,R^\star,H^\star)$. Thus, $(\hat A,\hat B,\hat Q,\hat R,\hat H)$ is consistent for $(A^\star,B^\star,Q^\star,R^\star,H^\star)$.
\end{proof}



\section{Numerical experiments}\label{sec:experiments}

This section evaluates CR-IOC from three complementary viewpoints. First, by varying the observation noise level and the number of demonstrations, we examine whether the staged reconstruction errors exhibit the robustness and consistency trends predicted by the perturbation analysis of Section~\ref{sec:noise}. Second, we compare CR-IOC with a matched baseline to isolate the benefit of exploiting finite-horizon gain variation in the dynamics stage, and we further test on a fixed dense $4\times 2$ example whether the resulting parameter improvements translate into more accurate closed-loop reconstruction. Third, we investigate whether the empirical diagnostics $(\hat c_X,\hat c_{AB},\hat c_{QRH})$ help identify stage-wise failure modes in deliberately ill-conditioned settings. We begin with a two-mass spring--damper system whose full $(n=4, m=2)$ model is written out explicitly and then scale the same construction to a longer chain.

\subsection{Experimental setup and benchmark models}\label{sec:experiments-protocol}

The reference system is the actuated two-mass spring--damper system shown in Fig.~\ref{fig:physical-benchmark}. Let $q=[q_1,q_2]^\top$ be the displacement vector and $u=[u_1,u_2]^\top$ the actuator force. With unit masses, ground spring/damping coefficients $(k_g,d_g)=(1.6,0.2)$, and coupling coefficients $(k_c,d_c)=(0.9,0.12)$, the second-order model is
\begin{equation}\label{eq:phys-second-order}
\ddot q + D_{\mathrm{phys}}\dot q + K_{\mathrm{phys}}q = u,
\end{equation}
where
\[
K_{\mathrm{phys}}=
\begin{bmatrix}
k_g+k_c & -k_c\\
-k_c & k_g+k_c
\end{bmatrix}
=
\begin{bmatrix}
2.5 & -0.9\\
-0.9 & 2.5
\end{bmatrix},
\qquad
D_{\mathrm{phys}}=
\begin{bmatrix}
d_g+d_c & -d_c\\
-d_c & d_g+d_c
\end{bmatrix}
=
\begin{bmatrix}
0.32 & -0.12\\
-0.12 & 0.32
\end{bmatrix}.
\]
Using the state $x=[q_1,q_2,\dot q_1,\dot q_2]^\top$, \eqref{eq:phys-second-order} becomes
\[
\dot{x}=A^{\star}x+B^{\star}u,
\]
where
\[
A^{\star}=
\begin{bmatrix}
0 & 0 & 1 & 0\\
0 & 0 & 0 & 1\\
-2.5 & 0.9 & -0.32 & 0.12\\
0.9 & -2.5 & 0.12 & -0.32
\end{bmatrix},
\qquad
B^{\star}=
\begin{bmatrix}
0 & 0\\
0 & 0\\
1 & 0\\
0 & 1
\end{bmatrix}
\]
with the diagonal cost weights
\[
Q^{\star}=\mathrm{diag}(1.182585,\,1.185247,\,0.920081,\,0.712088),
\]
\[
R^{\star}=\mathrm{diag}(1.097447,\,0.902553),
\qquad
H^{\star}=3Q^{\star}.
\]
The system in \eqref{eq:phys-second-order} is used as the $(n=4,m=2)$ reference case in the recoverability study and provides the physical interpretation for the conditioning discussion. The use of this example is motivated by the standard role of low-order structured systems in recent data-driven LQR studies, where they provide a transparent setting for separating estimation effects from synthesis effects before moving to larger-scale instances; see, e.g., the low-order locally coupled example in Dean et al.~\cite{dean2020sample} and the illustrative low-dimensional LQR examples in Mania et al.~\cite{mania2019certainty}. To examine the scaling, we also consider the analogous six-mass chain with actuators placed on masses $1$, $3$, $5$, and $6$, which yields $(n,m)=(12,4)$. Section~\ref{sec:experiments-e2e} then considers a fixed dense $4\times 2$ benchmark to assess closed-loop reconstruction in a setting that does not inherit the sparsity and symmetry of the mass--spring dynamics.

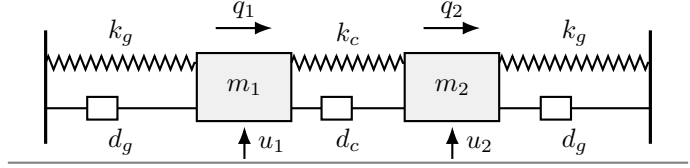
\begin{figure}[H]
\centering
\begin{tikzpicture}[x=1cm,y=1cm,>=Latex,thick]
  \draw[gray] (-0.5,0) -- (8.5,0);
  \draw[line width=1.2pt] (0,0.25) -- (0,1.75);
  \draw[line width=1.2pt] (8,0.25) -- (8,1.75);
  \draw[fill=gray!12] (2.0,0.55) rectangle (3.25,1.45);
  \draw[fill=gray!12] (4.75,0.55) rectangle (6.0,1.45);
  \node at (2.63,1.0) {$m_1$};
  \node at (5.38,1.0) {$m_2$};
  \draw[decorate,decoration={zigzag,segment length=4pt,amplitude=2.5pt}] (0,1.32) -- (2.0,1.32);
  \draw[decorate,decoration={zigzag,segment length=4pt,amplitude=2.5pt}] (3.25,1.32) -- (4.75,1.32);
  \draw[decorate,decoration={zigzag,segment length=4pt,amplitude=2.5pt}] (6.0,1.32) -- (8.0,1.32);
  \node[above] at (1.0,1.42) {$k_g$};
  \node[above] at (4.0,1.42) {$k_c$};
  \node[above] at (7.0,1.42) {$k_g$};
  \draw (0,0.72) -- (0.55,0.72);
  \draw (0.55,0.57) rectangle (0.95,0.87);
  \draw (0.95,0.72) -- (2.0,0.72);
  \draw (3.25,0.72) -- (3.65,0.72);
  \draw (3.65,0.57) rectangle (4.05,0.87);
  \draw (4.05,0.72) -- (4.75,0.72);
  \draw (6.0,0.72) -- (6.55,0.72);
  \draw (6.55,0.57) rectangle (6.95,0.87);
  \draw (6.95,0.72) -- (8.0,0.72);
  \node[below] at (1.0,0.58) {$d_g$};
  \node[below] at (4.0,0.58) {$d_c$};
  \node[below] at (7.0,0.58) {$d_g$};
  \draw[->] (2.63,0.05) -- (2.63,0.48);
  \draw[->] (5.38,0.05) -- (5.38,0.48);
  \node[right] at (2.68,0.24) {$u_1$};
  \node[right] at (5.43,0.24) {$u_2$};
  \draw[->] (2.25,1.78) -- (3.0,1.78);
  \draw[->] (5.0,1.78) -- (5.75,1.78);
  \node[above] at (2.63,1.78) {$q_1$};
  \node[above] at (5.38,1.78) {$q_2$};
\end{tikzpicture}
\caption{Reference two-mass spring--damper system used to anchor the experiments. Two actuated masses are connected to the ground and to each other by springs and dampers. The state is $x=[q_1,q_2,\dot q_1,\dot q_2]^\top$ and the input is $u=[u_1,u_2]^\top$. The medium-scale system is obtained by extending the same chain to six masses.}
\label{fig:physical-benchmark}
\end{figure}

All experiments use expert demonstrations generated from the finite-horizon LQR problem \eqref{eq:dyn}--\eqref{eq:lqr} under the assumptions of the paper, and the scale ambiguity is removed by \eqref{eq:scale-norm}. Expert trajectories are sampled on a uniform grid $t_k=k\Delta t$. Unless stated otherwise, additive Gaussian observation noise is applied to both states and inputs, process noise is absent, and the reported curves show medians together with interquartile ranges over $30$ Monte Carlo trials. The noise level is parameterized by a common signal-to-noise ratio (SNR). We define
\[
\mathrm{SNR}:=10\log_{10}\frac{P_j}{\sigma_j^2},\qquad j\in\{x,u\},
\]
where
\[
P_x:=\frac{1}{nN(L+1)}\sum_{i=1}^N\sum_{k=0}^L \|x_i(t_k)\|_2^2,
\qquad
P_u:=\frac{1}{mN(L+1)}\sum_{i=1}^N\sum_{k=0}^L \|u_i(t_k)\|_2^2,
\]
and set
\[
\sigma_x:=\sqrt{P_x}\,10^{-\mathrm{SNR}/20},
\qquad
\sigma_u:=\sqrt{P_u}\,10^{-\mathrm{SNR}/20}.
\]
Each noise vector is drawn as $v^x_{i,k}\sim\mathcal N(0,\sigma_x^2 I_n)$ and $v^u_{i,k}\sim\mathcal N(0,\sigma_u^2 I_m)$, independently across trajectories and time steps. The case $\mathrm{SNR}=\infty$ corresponds to the noiseless observations.

The proposed estimator is exactly Algorithm~\ref{alg:fhct-joint}. In the experiments below, the cost-stage semidefinite program is modeled in CVXPY~\cite{diamond2016cvxpy} and solved with Clarabel~\cite{goulart2024clarabel}. For the baseline comparison, we use a matched comparator denoted by \emph{SysID+IOC}. It shares the same denoising step, gain reconstruction, and cost-recovery SDP as CR-IOC and differs only in the dynamics stage: instead of recovering $(A,B)$ through the gain-variation Gramian \eqref{eq:AB-hat}, it fits the open-loop model $\dot x=Ax+Bu$ by least squares from the same noisy trajectories. This is intended as a controlled comparison in the usual ablation-study sense: all stages other than the use of finite-horizon gain-variation structure are held fixed. Thus, the comparison isolates whether exploiting that structure improves joint recovery. The target of this ablation is also consistent with the identifiability discussion in Geadah et al.~\cite{geadah2024inferring}: in a discrete-time stochastic setting that closed-loop data under infinite-horizon operation do not separate the open-loop dynamics unless additional structure is used, whereas finite horizons can restore identifiability. In this sense, \emph{SysID+IOC} plays the role of a comparator that does not exploit the finite-horizon temporal structure emphasized by our theory. Accordingly, \emph{SysID+IOC} should be read as a controlled ablation of the dynamics stage rather than as an exhaustive benchmark against all possible system-identification pipelines.

We report the stage-wise reconstruction errors
\[
\mathrm{Err}_{AB}:=\|\hat A-A^\star\|_F^2+\|\hat B-B^\star\|_F^2,
\qquad
\mathrm{Err}_{QR}:=\|\hat Q-Q^\star\|_F^2+\|\hat R-R^\star\|_F^2,
\]
\[
\mathrm{Err}_{K}:=\|\hat K-K^\star\|_F^2,
\qquad
\mathrm{NMSE}_{x,u}:=
\frac{\sum_{k=0}^{L}\big\|[\hat x_k;\hat u_k]-[x_k^\star;u_k^\star]\big\|_2^2}
{\sum_{k=0}^{L}\big\|[x_k^\star;u_k^\star]\big\|_2^2}.
\]
Because $H=\alpha Q$ under Assumption~\ref{ass:terminal-prop} and the normalization \eqref{eq:scale-norm} removes the positive scaling ambiguity, the cost stage is evaluated by the raw Frobenius error $\mathrm{Err}_{QR}$. The empirical conditioning diagnostics $(\hat c_X,\hat c_{AB},\hat c_{QRH})$ are computed from \eqref{eq:cX-hat}, \eqref{eq:GK-hat}, and \eqref{eq:cQRH-hat}. 

\subsection{Recoverability under sample size and observation noise}\label{sec:experiments-recoverability}

We first study perturbation trends on the physical two-mass system of Section~\ref{sec:experiments-protocol} and on its six-mass extension. Both systems use $T=2\,\mathrm{s}$, $\Delta t=0.01\,\mathrm{s}$, $N\in\{30,100,300,1000\}$, and $\mathrm{SNR}\in\{\infty,40,30,20\}\,\mathrm{dB}$. Unless varied explicitly, the default settings are $\Delta=0.04\,\mathrm{s}$ and $\mathrm{BW}=0.06 \,\mathrm{s}$. 

\begin{figure}[t]
\centering
\includegraphics[width=0.98\linewidth]{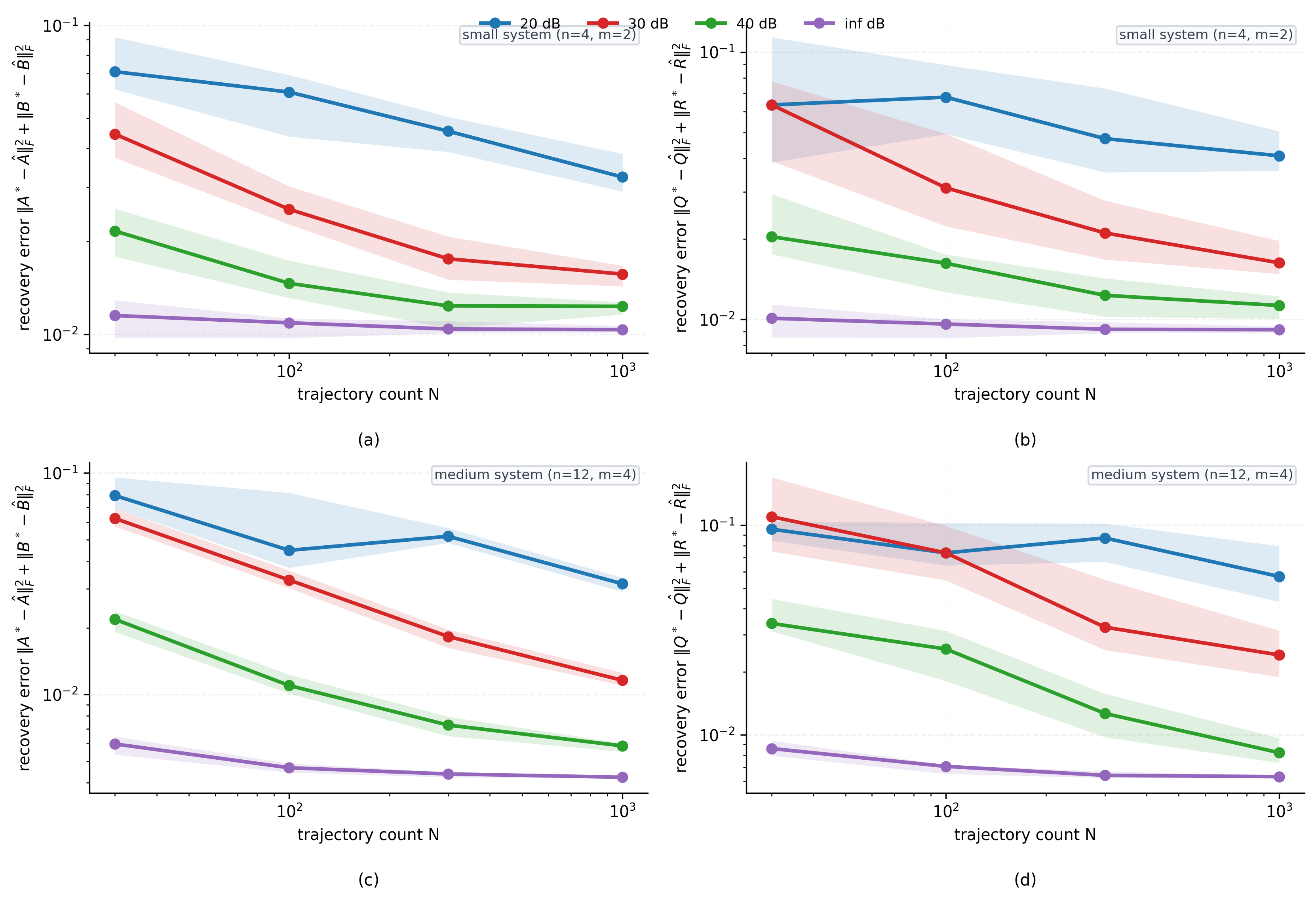}
\caption{Recoverability sweep on the mass--spring--damper family. Fig.~\ref{fig:recoverability-sweep}(a)--(b) show the physical two-mass system $(n=4,m=2)$, and Fig.~\ref{fig:recoverability-sweep}(c)--(d) show its six-mass extension $(n=12,m=4)$. Curves report medians over $30$ trials and shaded bands show interquartile ranges. Larger demonstration sets and higher observation SNR generally reduce the dynamics and cost errors.}
\label{fig:recoverability-sweep}
\end{figure}

Figure~\ref{fig:recoverability-sweep} shows the expected overall trend: for both benchmarks, larger demonstration sets and cleaner measurements generally reduce both $\mathrm{Err}_{AB}$ and $\mathrm{Err}_{QR}$. At $30$ dB, for example, the physical two-mass system decreases from $4.45\times 10^{-2}$ and $6.35\times 10^{-2}$ at $N=30$ to $1.57\times 10^{-2}$ and $1.63\times 10^{-2}$ at $N=1000$. On the six-mass extension the corresponding errors decrease from $6.24\times 10^{-2}$ and $1.10\times 10^{-1}$ to $1.16\times 10^{-2}$ and $2.41\times 10^{-2}$. The medium-scale curves are not strictly monotone at every SNR, but the overall decrease from small to large $N$ remains clear, which is reasonable for finite-sample Monte Carlo medians in the hardest noise level.

These trends are consistent with the staged perturbation analysis. Increasing $N$ enriches the sampled state data entering the gain stage, while cleaner observations reduce the perturbations propagated through the downstream dynamics and cost stages; see Lemma~\ref{lem:K-ident} and Theorems~\ref{thm:K-perturb}, \ref{thm:AB-perturb}, and \ref{thm:QRH-perturb}. The larger chain also shows that the same qualitative picture is not confined to the $4\times 2$ benchmark. The $\infty$ dB curves flatten near a small numerical floor, which is consistent with residual finite-grid and solver tolerances rather than with a qualitative failure of the reconstruction pipeline.

To examine the two principal tuning parameters of CR-IOC under a fixed noise level, we also report a sensitivity study on the same mass--spring--damper family. The setting is $T=2\,\mathrm{s}$, $\Delta t=0.01\,\mathrm{s}$, $N=300$, $\mathrm{SNR}=20\,\mathrm{dB}$, no process noise, and $30$ Monte Carlo trials. Table~\ref{tab:delta-sensitivity} varies the increment window $\Delta$ while fixing the denoising bandwidth at $\BW=0.06\,\mathrm{s}$, and Table~\ref{tab:bw-sensitivity} varies $\BW$ while fixing $\Delta=0.04\,\mathrm{s}$. In both tables, the entries are median values of $\mathrm{Err}_{AB}$ and $\mathrm{Err}_{QR}$.

\begin{table}[h]
\centering
\small
\setlength{\tabcolsep}{4.5pt}
\caption{Sensitivity to the increment window \(\Delta\) under \(T=2\,\mathrm{s}\), \(\Delta t=0.01\,\mathrm{s}\), \(N=300\), \(\mathrm{SNR}=20\,\mathrm{dB}\), \(\BW=0.06\,\mathrm{s}\), and \(30\) Monte Carlo trials.}
\label{tab:delta-sensitivity}
\begin{tabular}{@{}c cc cc@{}}
\toprule
 & \multicolumn{2}{c}{\((n,m)=(4,2)\)} & \multicolumn{2}{c}{\((n,m)=(12,4)\)} \\
\cmidrule(lr){2-3}\cmidrule(lr){4-5}
\(\Delta\) (s) & \(\mathrm{Err}_{AB}\) & \(\mathrm{Err}_{QR}\) & \(\mathrm{Err}_{AB}\) & \(\mathrm{Err}_{QR}\) \\
\midrule
0.02 & \(4.77 \times 10^{-2}\) & \(5.31 \times 10^{-2}\) & \(5.53 \times 10^{-2}\) & \(9.28 \times 10^{-2}\) \\
0.04 & \(5.42 \times 10^{-2}\) & \(4.83 \times 10^{-2}\) & \(5.86 \times 10^{-2}\) & \(9.20 \times 10^{-2}\) \\
0.06 & \(6.10 \times 10^{-2}\) & \(5.62 \times 10^{-2}\) & \(5.78 \times 10^{-2}\) & \(8.99 \times 10^{-2}\) \\
0.08 & \(6.96 \times 10^{-2}\) & \(6.41 \times 10^{-2}\) & \(6.33 \times 10^{-2}\) & \(9.11 \times 10^{-2}\) \\
0.10 & \(6.95 \times 10^{-2}\) & \(7.20 \times 10^{-2}\) & \(6.60 \times 10^{-2}\) & \(1.029 \times 10^{-1}\) \\
0.12 & \(7.56 \times 10^{-2}\) & \(6.52 \times 10^{-2}\) & \(7.12 \times 10^{-2}\) & \(9.16 \times 10^{-2}\) \\
\bottomrule
\end{tabular}
\end{table}
Table~\ref{tab:delta-sensitivity} indicates that the reconstruction is only moderately sensitive to the increment window within the tested range. For the smaller system $(n,m)=(4,2)$, increasing $\Delta$ beyond the default value leads to a visible increase in both $\mathrm{Err}_{AB}$ and $\mathrm{Err}_{QR}$. For the larger system $(n,m)=(12,4)$, the dependence is milder: $\mathrm{Err}_{AB}$ increases gradually for larger windows, while $\mathrm{Err}_{QR}$ remains comparatively flat except for a mild degradation around the largest windows. Overall, these trends are compatible with the local $A_c(\cdot)$ reconstruction analysis in Section~\ref{sec:noise}, where increasing $\Delta$ reduces short-window noise amplification but increases discretization and sampling bias.

Table~\ref{tab:bw-sensitivity} shows an even milder dependence on the denoising bandwidth. Over the range $\BW\in[0.015,0.09]$, the changes in both $\mathrm{Err}_{AB}$ and $\mathrm{Err}_{QR}$ remain modest, and $\BW=0.06\,\mathrm{s}$ lies inside a stable plateau rather than at the boundary of the tested range. Very small bandwidths can slightly increase errors in some entries, whereas larger bandwidths do not provide a systematic improvement. Thus, the default bandwidth does not appear to rely on fine tuning.

\begin{table}[h]
\centering
\small
\setlength{\tabcolsep}{4.5pt}
\caption{Sensitivity to the denoising bandwidth \(\BW\) under \(T=2\,\mathrm{s}\), \(\Delta t=0.01\,\mathrm{s}\), \(N=300\), \(\mathrm{SNR}=20\,\mathrm{dB}\), \(\Delta=0.04\,\mathrm{s}\), and \(30\) Monte Carlo trials.}
\label{tab:bw-sensitivity}
\begin{tabular}{@{}c cc cc@{}}
\toprule
 & \multicolumn{2}{c}{\((n,m)=(4,2)\)} & \multicolumn{2}{c}{\((n,m)=(12,4)\)} \\
\cmidrule(lr){2-3}\cmidrule(lr){4-5}
\(\BW\) (s) & \(\mathrm{Err}_{AB}\) & \(\mathrm{Err}_{QR}\) & \(\mathrm{Err}_{AB}\) & \(\mathrm{Err}_{QR}\) \\
\midrule
0.015 & \(5.66 \times 10^{-2}\) & \(5.54 \times 10^{-2}\) & \(5.59 \times 10^{-2}\) & \(9.59 \times 10^{-2}\) \\
0.030 & \(4.84 \times 10^{-2}\) & \(5.97 \times 10^{-2}\) & \(5.80 \times 10^{-2}\) & \(9.68 \times 10^{-2}\) \\
0.045 & \(5.50 \times 10^{-2}\) & \(5.24 \times 10^{-2}\) & \(6.03 \times 10^{-2}\) & \(9.67 \times 10^{-2}\) \\
0.060 & \(5.49 \times 10^{-2}\) & \(5.25 \times 10^{-2}\) & \(5.67 \times 10^{-2}\) & \(9.22 \times 10^{-2}\) \\
0.090 & \(5.74 \times 10^{-2}\) & \(5.08 \times 10^{-2}\) & \(6.58 \times 10^{-2}\) & \(9.38 \times 10^{-2}\) \\
\bottomrule
\end{tabular}
\end{table}

\subsection{Comparison with a matched baseline under noisy observations}\label{sec:experiments-baseline}

We next fix the six-mass chain and compare CR-IOC with the matched baseline \emph{SysID+IOC} under noisy observations. The experiment uses $T=2\,\mathrm{s}$, $\Delta t=0.01\,\mathrm{s}$, $N=1200$ demonstrations, and $\mathrm{SNR}\in\{10,15,20,30,40\}\,\mathrm{dB}$, again with $30$ Monte Carlo trials per setting.

\begin{figure}[t]
\centering
\includegraphics[width=0.98\linewidth]{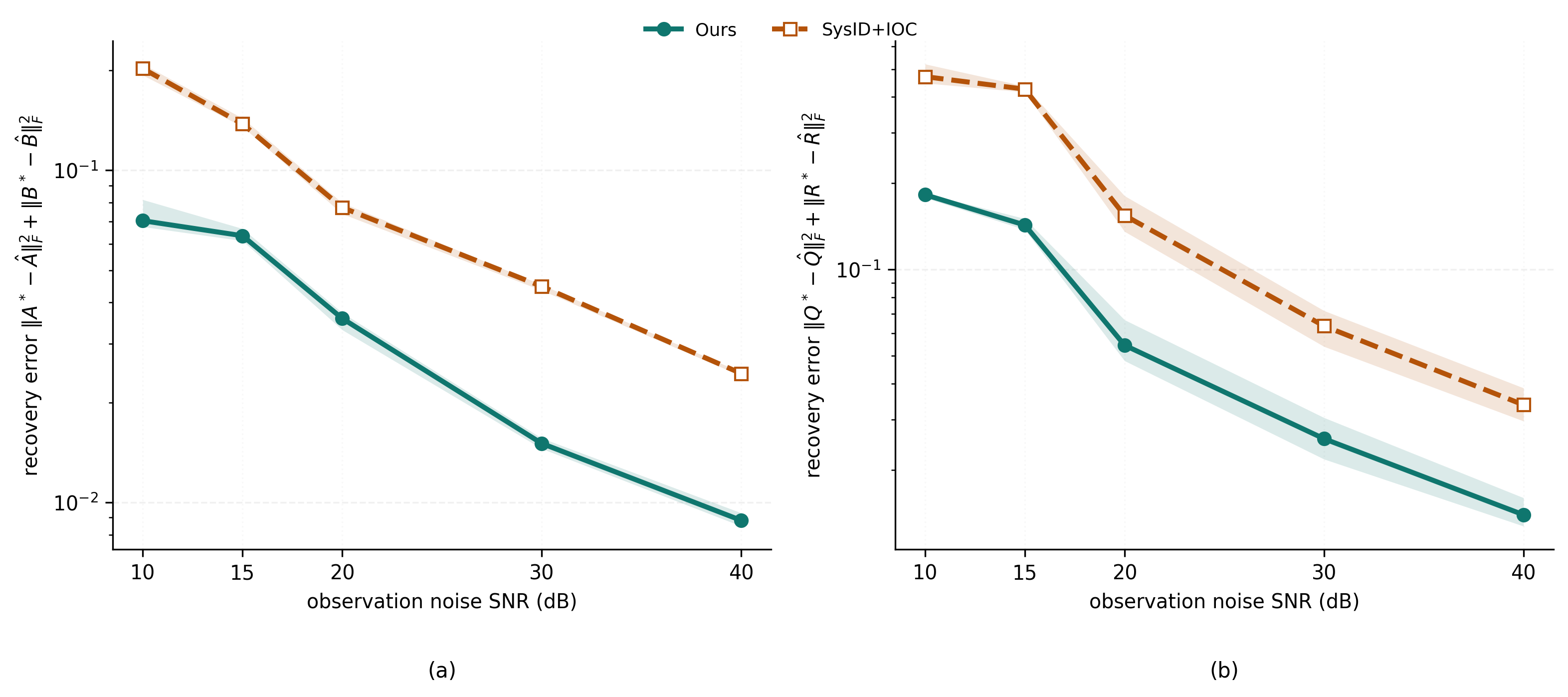}
\caption{Matched noisy comparison on the six-mass chain $(n=12,m=4)$. Both methods use the same denoising, gain reconstruction, and cost-recovery modules; they differ only in the dynamics stage. CR-IOC remains below \emph{SysID+IOC} across the full SNR range in both $\mathrm{Err}_{AB}$ and $\mathrm{Err}_{QR}$.}
\label{fig:baseline-comparison}
\end{figure}

CR-IOC improves both metrics at every tested SNR level. At $20$ dB, for example, the median dynamics error decreases from $7.72\times 10^{-2}$ to $3.58\times 10^{-2}$ and the median cost error decreases from $1.55\times 10^{-1}$ to $5.45\times 10^{-2}$, corresponding to reductions of about $54\%$ and $65\%$, respectively. The same advantage is visible at both ends of the noise range: at $10$ dB the reductions are about $65\%$ in $\mathrm{Err}_{AB}$ and $61\%$ in $\mathrm{Err}_{QR}$, while at $40$ dB they remain about $64\%$ and $59\%$. Across the tested range, the ranking never inverts as the observations become cleaner, and the gap remains visible even when the baseline receives relatively mild noise.

This pattern illustrates the main empirical advantage targeted by CR-IOC. The two methods use the same denoised trajectories and the same final SDP; the only algorithmic difference is whether the finite-horizon variation of $K(t)$ is used to separate $A$ and $B$. The persistence of the gap up to $40$ dB therefore suggests that the improvement is due in part to this structural information rather than solely to additional smoothing or regularization. This interpretation is consistent with Theorem~\ref{thm:AB-ident}: finite-horizon gain variation carries information that separates $A$ and $B$, whereas the open-loop least-squares baseline does not exploit that structure. More accurate recovery of $(A,B)$ then benefits the cost stage through the coupling captured by Theorem~\ref{thm:joint-ident} and quantified in Theorem~\ref{thm:QRH-perturb}. 

\subsection{Closed-loop reconstruction on a fixed dense $4\times 2$ benchmark}\label{sec:experiments-e2e}

To complement the structured mass-spring examples, we next consider a fixed dense $4\times 2$ benchmark. This example is used to assess closed-loop reconstruction in a setting that does not inherit the sparsity and symmetry of the mass-spring dynamics. The tuple $(A_1^{\star},B_1^{\star},Q_1^{\star},R_1^{\star},H_1^{\star})$ was generated once at random subject to Assumptions~\ref{ass:standing} and \ref{ass:terminal-prop}, and then held fixed throughout the Monte Carlo study. The reference system is
\[
A_1^{\star}=
\begin{bmatrix}
0.321246 & 0.710597 & -1.021317 & -0.055186\\
0.405488 & 0.440857 & 0.261515 & 0.598847\\
0.115983 & 0.220507 & -0.028505 & -0.429543\\
-0.338652 & 0.151834 & -0.232078 & 0.408621
\end{bmatrix},
\qquad
B_1^{\star}=
\begin{bmatrix}
-0.391870 & 0.049920\\
-0.165190 & -0.167969\\
0.062616 & -0.225711\\
0.296923 & -0.048893
\end{bmatrix},
\]
with diagonal weights
\[
Q_1^{\star}=\mathrm{diag}(0.632127,\,0.795285,\,1.404545,\,1.168044),
\qquad
R_1^{\star}=\mathrm{diag}(1.215864,\,0.784136),
\qquad
H_1^{\star}=Q_1^{\star}.
\]
The benchmark uses $T=2.0\,\mathrm{s}$, $\Delta t=0.01\,\mathrm{s}$, $N=16$, and $30$ dB observation noise, with $30$ Monte Carlo trials.

\begin{figure}[t]
\centering
\includegraphics[width=0.98\linewidth]{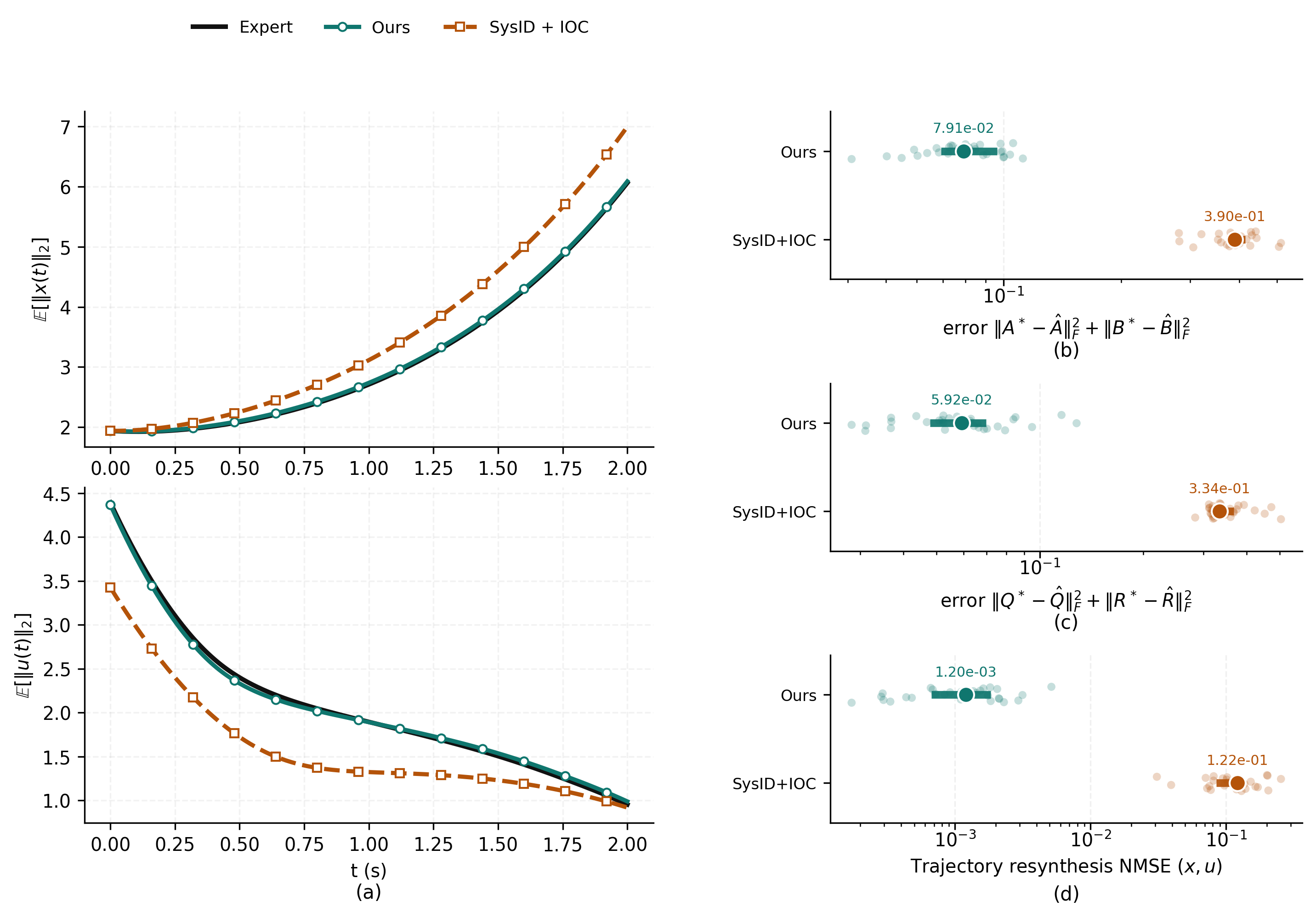}
\caption{Closed-loop reconstruction on the fixed dense reference system $(A_1^{\star},B_1^{\star})$. The left column reports aggregated state-norm and input-norm profiles for the expert and reconstructed closed loops, and the right column summarizes the Monte Carlo distributions of $\mathrm{Err}_{AB}$, $\mathrm{Err}_{QR}$, and $\mathrm{NMSE}_{x,u}$.}
\label{fig:resynthesis-benchmark}
\end{figure}

\begin{table}[t]
\centering
\small
\caption{Closed-loop reconstruction on the fixed dense $4\times 2$ benchmark (median over $30$ trials).}
\label{tab:resynthesis-benchmark}
\begin{tabular}{@{}lcccc@{}}
\toprule
Method & $\mathrm{Err}_{AB}$ & $\mathrm{Err}_{QR}$ & $\mathrm{NMSE}_{x,u}$& runtime (s)\\
\midrule
CR-IOC & $7.91\times 10^{-2}$ & $5.92\times 10^{-2}$ & $1.20\times 10^{-3}$ & 0.38\\
\emph{SysID+IOC} & $3.90\times 10^{-1}$ & $3.34\times 10^{-1}$ & $1.22\times 10^{-1}$ & 0.27\\
\bottomrule
\end{tabular}
\end{table}

Table~\ref{tab:resynthesis-benchmark} shows that the CR-IOC algorithm achieves a more pronounced performance improvement in closed-loop trajectory recovery than in the recovery of dynamics parameters and cost function parameters.
 Relative to \emph{SysID+IOC}, the median rollout mismatch $\mathrm{NMSE}_{x,u}$ decreases from $1.22\times 10^{-1}$ to $1.20\times 10^{-3}$, whereas $\mathrm{Err}_{AB}$ and $\mathrm{Err}_{QR}$ are reduced by factors of approximately $4.9$ and $5.6$, respectively. This is also reflected in Fig.~\ref{fig:resynthesis-benchmark}: the trajectories reconstructed from CR-IOC track the aggregate state- and input-norm profiles of the expert much more closely over the full horizon, whereas \emph{SysID+IOC} shows clear transient mismatches in both quantities. Hence, the advantage of improved joint recovery is not merely parametric; it yields a substantially smaller closed-loop mismatch of the finite-horizon closed-loop behavior induced by the recovered system and cost.

This behavior is qualitatively consistent with the error-propagation picture in Theorem~\ref{thm:e2e}. The theorem controls the staged reconstruction errors in $K(\cdot)$, $A_c(\cdot)$, $(A,B)$, and the normalized cost parameters under the admissible perturbation conditions. Although it does not directly bound the rollout metric $\mathrm{NMSE}_{x,u}$, the finite-horizon closed-loop trajectories depend continuously on these recovered quantities over a bounded time interval. Table~\ref{tab:resynthesis-benchmark} also shows that, on this small benchmark, the two methods have runtimes of the same order in the present implementation; thus, the behavior-level improvement is not tied to a large runtime gap in this setting.

\subsection{Conditioning diagnosis}\label{sec:experiments-conditioning}

Finally, we probe the theory by constructing three diagnostic settings on the fixed dense $4\times2$ benchmark introduced in Section~\ref{sec:experiments-e2e}. This is a targeted conditioning study rather than a broad Monte Carlo sweep. Thus, we report means over $30$ Monte Carlo trials for each of the three settings. Setting R1 is the informative reference case with $T_{\mathrm{obs}}=T_{\mathrm{expert}}=2.0\,\mathrm{s}$ and $N=16$. Setting R2 shortens the observation window to $T_{\mathrm{obs}}=0.2\,\mathrm{s}$ while keeping $T_{\mathrm{expert}}=4.0\,\mathrm{s}$, uses only $N=8$ demonstrations, and reduces the initial-state scale to $0.2$, thereby degrading both excitation and gain variation. Setting R3 keeps $T_{\mathrm{obs}}=T_{\mathrm{expert}}=2.0\,\mathrm{s}$ and uses $N=12$, but trims the tail of the horizon in the cost stage so that only the first $3\%$ of the interval is used to assemble the cost operator. All three settings use $\Delta t=0.01\,\mathrm{s}$ and $30$ dB observation noise.

\begin{figure}[t]
\centering
\includegraphics[width=0.98\linewidth]{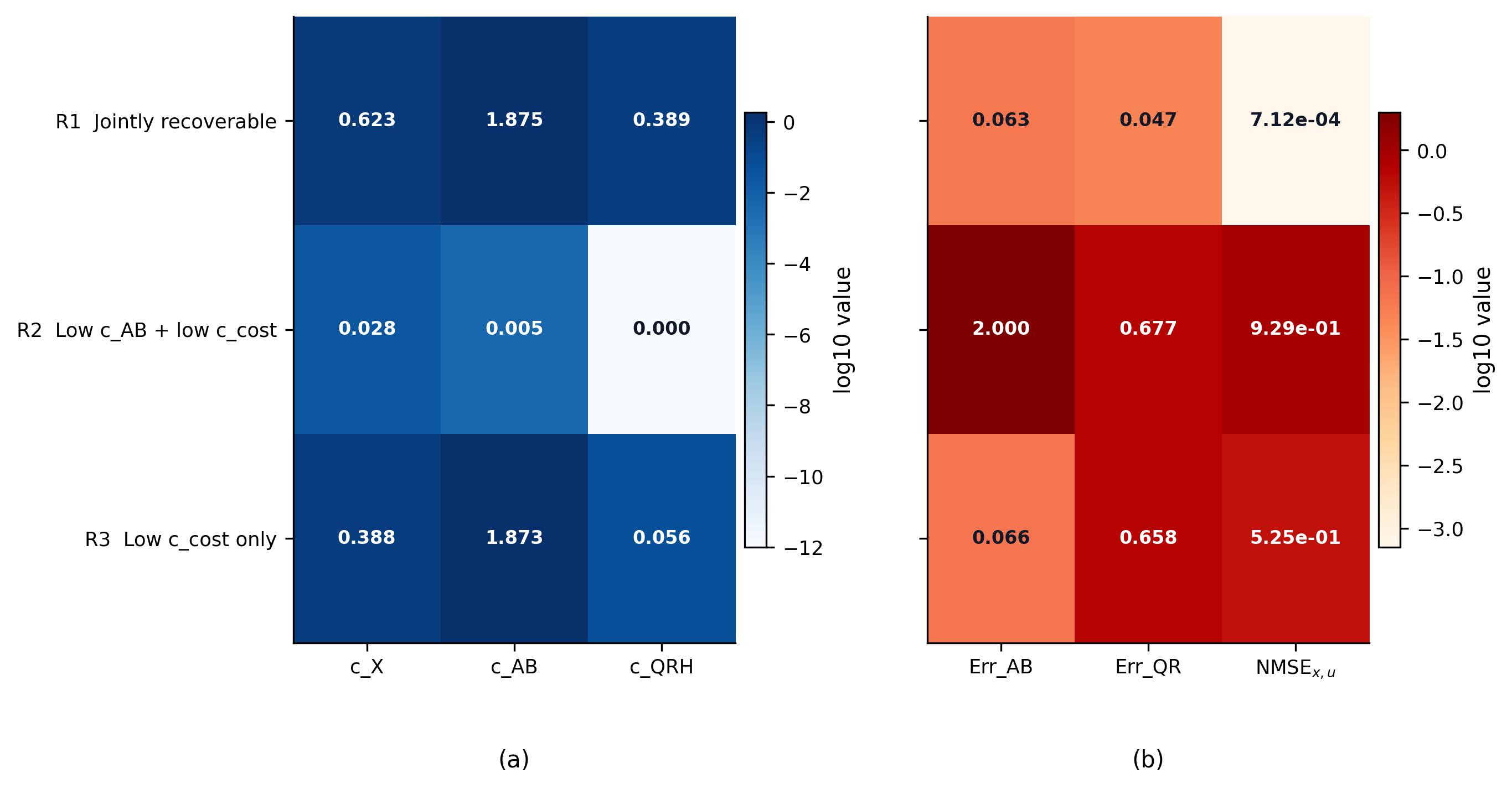}
\caption{Diagnostic conditioning study on the fixed $4\times 2$ reference system. Fig.~\ref{fig:conditioning-regimes}(a) reports the means of the empirical conditioning indices for each setting, and Fig.~\ref{fig:conditioning-regimes}(b) reports the corresponding mean recovery errors and rollout mismatch. R1 is a recoverable case, R2 is jointly ill-conditioned, and R3 primarily degrades the cost stage.}
\label{fig:conditioning-regimes}
\end{figure}

Figure~\ref{fig:conditioning-regimes} shows the stage-wise pattern anticipated by the theory. In R1, all three empirical indices remain clearly positive, and the mean errors are small. In R2, the full pipeline becomes poorly conditioned: $\hat c_X= 0.028$, $\hat c_{AB}= 5.0\times 10^{-3}$, and $\hat c_{QRH}$ is numerically zero at the displayed precision, while the corresponding mean errors increase to $\mathrm{Err}_{AB}=2.00$, $\mathrm{Err}_{QR}=0.677$, and $\mathrm{NMSE}_{x,u}=9.29\times 10^{-1}$. In R3, by contrast, $\hat c_X= 0.388$ and $\hat c_{AB}= 1.873$ remain comparable to the recoverable setting, and the dynamics error stays at $0.066$, but the reduction in $\hat c_{QRH}$ to $0.056$ is accompanied by strong degradation in both $\mathrm{Err}_{QR}$ and $\mathrm{NMSE}_{x,u}$.

This stage-selective degradation is the practical value of the conditioning indices. Theorem~\ref{thm:joint-ident} shows that positivity of both the
dynamics and cost indices guarantees joint identifiability; Theorem~\ref{thm:AB-perturb} explains why the dynamics stage remains stable when $\hat c_{AB}$ stays away from zero; and Theorem~\ref{thm:QRH-perturb} shows why the cost stage can still fail when only $\hat c_{QRH}$ deteriorates. For experimental design, the diagnostics are informative rather than merely retrospective. A setting like R2 is unlikely to be repaired by solver tuning alone and instead points to insufficient excitation, insufficient observation length, or both. A setting like R3 points more specifically to a cost-stage bottleneck, for which retaining more late-horizon information or modifying the cost parametrization are natural remedies. Thus, the empirical indices help indicate which aspect of the data-collection design should be changed.

Taken together, the experiments support the main qualitative claims of the paper: finite-horizon gain variation can make joint recovery possible, the empirical conditioning indices help explain when that recovery is numerically stable, and the resulting parameter improvements can be meaningful at the behavioral level. The experiments remain limited to full state–input observations and deterministic dynamics without process noise. Therefore, extending the same diagnostic framework to partial observations and stochastic settings remains an important direction for future work.

\section{Conclusion}\label{sec:conclusion}
This paper studied finite-horizon continuous-time inverse LQR from state-input closed-loop trajectories when both the system dynamics and the quadratic cost are unknown. When sufficiently rich, finite-horizon gain variation provides the information needed to separate the constant open-loop pair from the closed-loop dynamics matrix. Within the admissible class analyzed in the paper, this leads to joint identifiability conditions for the unknown dynamics and cost.

We quantified this structure through three conditioning indices. The index $c_X$ measures state richness and governs recovery of $K(\cdot)$ and $A_c(\cdot)$ from data. The index $c_{AB}$ measures the richness of gain variation needed to separate $(A,B)$. The index $c_{QRH}$ measures injectivity of the structured cost operator after scale removal. These quantities unify the paper: they appear in the identifiability theory, in the sampled-data reconstruction method, and in the perturbation bounds.

On this basis, We proposed CR-IOC, a conditioning-aware sampled-data reconstruction method built on temporal denoising and Richardson bias-cancellation. The method performs local closed-loop reconstruction of the dynamics, followed by closed-form recovery of $(A,B)$ and convex recovery of the quadratic weights. We established non-asymptotic perturbation bounds and consistency under sampling and sub-Gaussian observation noise, showing how the same conditioning indices govern sensitivity throughout the reconstruction pipeline.

Because the indices are computable from data through their empirical counterparts, the theory also yields practical diagnostics for data screening, horizon selection, and numerical regularization. The present analysis still assumes full state-input trajectories and focuses on deterministic LTI systems with quadratic costs within the structured family considered in the paper. Extending the operator framework to partial observations, process noise, output-feedback settings, and broader classes of optimal control objectives remains an important direction for future work.

\bibliographystyle{IEEEtran}
\bibliography{reference}

@book{lewis2012optimal,
  title={Optimal control},
  author={Lewis, Frank L and Vrabie, Draguna and Syrmos, Vassilis L},
  year={2012},
  publisher={John Wiley \& Sons}
}

@article{ab2020inverse,
  title={From inverse optimal control to inverse reinforcement learning: A historical review},
  author={Ab Azar, Nematollah and Shahmansoorian, Aref and Davoudi, Mohsen},
  journal={Annual Reviews in Control},
  volume={50},
  pages={119--138},
  year={2020},
  publisher={Elsevier}
}

@article{adams2022survey,
  title={A survey of inverse reinforcement learning},
  author={Adams, Stephen and Cody, Tyler and Beling, Peter A},
  journal={Artificial Intelligence Review},
  volume={55},
  number={6},
  pages={4307--4346},
  year={2022},
  publisher={Springer}
}

@article{chan2025inverse,
  title={Inverse optimization: Theory and applications},
  author={Chan, Timothy CY and Mahmood, Rafid and Zhu, Ian Yihang},
  journal={Operations Research},
  volume={73},
  number={2},
  pages={1046--1074},
  year={2025},
  publisher={INFORMS}
}

@article{li2020ctiqoc,
  title        = {Continuous-time inverse quadratic optimal control problem},
  author       = {Li, Yibei and Yao, Yu and Hu, Xiaoming},
  journal      = {Automatica},
  volume       = {117},
  pages        = {108977},
  year         = {2020},
  doi          = {10.1016/j.automatica.2020.108977}
}

@article{cheng2026ddioc,
  title        = {Data-driven inverse optimal control for linear quadratic tracking with unknown target states},
  author       = {Cheng, Renshuo and Yu, Chengpu and Li, Yao},
  journal      = {Automatica},
  volume       = {185},
  pages        = {112822},
  year         = {2026},
  doi          = {10.1016/j.automatica.2026.112822}
}

@article{cao2025inverse,
  title={Inverse Continuous-Time Linear Quadratic Regulator: From Control Cost Matrix to Entire Cost Reconstruction},
  author={Cao, Yuexin and Li, Yibei and Zou, Zhuo and Hu, Xiaoming},
  journal={arXiv:2510.04083},
  year={2025}
}

@book{anderson2007optimal,
  title        = {Optimal Control: Linear Quadratic Methods},
  author       = {Anderson, Brian D. O. and Moore, John B.},
  publisher    = {Dover Publications},
  year         = {2007},
  isbn         = {9780486457666}
}

@article{cao2025differential,
  title={A differential dynamic programming framework for inverse reinforcement learning},
  author={Cao, Kun and Xu, Xinhang and Jin, Wanxin and Johansson, Karl H and Xie, Lihua},
  journal={IEEE Transactions on Robotics},
  year={2025},
  publisher={IEEE}
}

@inproceedings{geadah2024inferring,
  title={Inferring system and optimal control parameters of closed-loop systems from partial observations},
  author={Geadah, Victor and Arbelaiz, Juncal and Ritz, Harrison and Daw, Nathaniel D and Cohen, Jonathan D and Pillow, Jonathan W},
  booktitle={2024 IEEE 63rd Conference on Decision and Control (CDC)},
  pages={8006--8013},
  year={2024},
  organization={IEEE}
}

@article{zhang2019lqr,
  title   = {Inverse optimal control for discrete-time finite-horizon linear quadratic regulators},
  author  = {Zhang, Han and Umenberger, Jack and Hu, Xiaoming},
  journal = {Automatica},
  volume  = {110},
  pages   = {108593},
  year    = {2019},
  doi     = {10.1016/j.automatica.2019.108593}
}

@article{yu2021sysid,
  title={System identification approach for inverse optimal control of finite-horizon discrete-time {LQR}},
  author={Yu, Chengpu and Gao, Zhe and Li, Yao},
  journal={Automatica},
  volume={129},
  pages={109636},
  year={2021},
  month={July},
  doi={10.1016/j.automatica.2021.109636}
}

@article{garrabe2025convexddioc,
  title   = {On convex data-driven inverse optimal control for nonlinear, non-stationary and stochastic systems},
  author  = {Garrabe, Emiland and Jesawada, Hozefa and Del Vecchio, Carmen and Russo, Giovanni},
  journal = {Automatica},
  volume  = {173},
  pages   = {112015},
  year    = {2025},
  month   = {March},
  doi     = {10.1016/j.automatica.2024.112015}
}

@article{mayne2000mpc,
  title   = {Constrained model predictive control: Stability and optimality},
  author  = {Mayne, David Q. and Rawlings, James B. and Rao, Christopher V. and Scokaert, Pierre O. M.},
  journal = {Automatica},
  volume  = {36},
  number  = {6},
  pages   = {789--814},
  year    = {2000},
  doi     = {10.1016/S0005-1098(99)00214-9}
}

@book{rawlings2009mpc,
  title     = {Model Predictive Control: Theory and Design},
  author    = {Rawlings, James B. and Mayne, David Q.},
  publisher = {Nob Hill Publishing},
  year      = {2009},
  isbn      = {9780975937709}
}

@article{kalman1964linear,
  title={When is a linear control system optimal?},
  author={R.E.Kalman},
  journal={Journal of Basic Engineering},
  year={1964},
  volume={86},
  number={1},
  pages={51--60},
}

@book{boyd1994linear,
  title={Linear matrix inequalities in system and control theory},
  author={Boyd, Stephen and El Ghaoui, Laurent and Feron, Eric and Balakrishnan, Venkataramanan},
  year={1994},
  publisher={Society for Industrial and Applied Mathematics}
}

@inproceedings{molloy2016discrete,
  title={Discrete-time inverse optimal control with partial-state information: A soft-optimality approach with constrained state estimation},
  author={Molloy, Timothy L and Tsai, Dorian and Ford, Jason J and Perez, Tristan},
  booktitle={2016 IEEE 55th Conference on Decision and Control (CDC)},
  pages={1926--1932},
  year={2016},
}

@article{zhang2019inverse,
  title={Inverse optimal control for discrete-time finite-horizon linear quadratic regulators},
  author={Zhang, Han and Umenberger, Jack and Hu, Xiaoming},
  journal={Automatica},
  volume={110},
  pages={108593},
  year={2019},
  publisher={Elsevier}
}

@inproceedings{zhang2022statistically,
  title={Statistically consistent inverse optimal control for linear-quadratic tracking with random time horizon},
  author={Zhang, Han and Ringh, Axel and Jiang, Weihan and Li, Shaoyuan and Hu, Xiaoming},
  booktitle={2022 41st Chinese Control Conference (CCC)},
  pages={1515--1522},
  year={2022},
}

@article{qu2024control,
  author={Qu, Chendi and He, Jianping and Duan, Xiaoming},
  journal={IEEE Transactions on Automatic Control}, 
  title={Control Law Learning Based on {LQR} Reconstruction With Inverse Optimal Control}, 
  year={2025},
  volume={70},
  number={2},
  pages={1350-1357}
}

@article{qu20243dioc,
  title={{3DIOC}: Direct Data-Driven Inverse Optimal Control for {LTI} Systems},
  author={Qu, Chendi and He, Jianping and Duan, Xiaoming},
  journal={arXiv:2409.10884},
  year={2024}
}

@article{xue2021inverse,
  title={Inverse reinforcement {Q}-learning through expert imitation for discrete-time systems},
  author={Xue, Wenqian and Lian, Bosen and Fan, Jialu and Kolaric, Patrik and Chai, Tianyou and Lewis, Frank L},
  journal={IEEE Transactions on Neural Networks and Learning Systems},
  volume={34},
  number={5},
  pages={2386--2399},
  year={2021},
  publisher={IEEE}
}

@article{jin2020pontryagin,
  title={Pontryagin differentiable programming: An end-to-end learning and control framework},
  author={Jin, Wanxin and Wang, Zhaoran and Yang, Zhuoran and Mou, Shaoshuai},
  journal={Advances in Neural Information Processing Systems},
  volume={33},
  pages={7979--7992},
  year={2020}
}

@article{zhang2023inverse,
  title={Inverse linear-quadratic discrete-time finite-horizon optimal control for indistinguishable homogeneous agents: A convex optimization approach},
  author={Zhang, Han and Ringh, Axel},
  journal={Automatica},
  volume={148},
  pages={110758},
  year={2023},
  publisher={Elsevier}
}

@article{li2024inverse,
  title={Inverse Kalman filtering problems for discrete-time systems},
  author={Li, Yibei and Wahlberg, Bo and Hu, Xiaoming and Xie, Lihua},
  journal={Automatica},
  volume={163},
  pages={111560},
  year={2024},
  publisher={Elsevier}
}

@inproceedings{karg2024bi,
  title={Bi-level-based inverse stochastic optimal control},
  author={Karg, Philipp and Hess, Manuel and Varga, Balint and Hohmann, S{\"o}ren},
  booktitle={2024 European Control Conference (ECC)},
  pages={537--544},
  year={2024},
}

@article{jean2024inverse,
  title={Inverse optimal control problem in the non autonomous linear-quadratic case},
  author={Jean, Fr{\'e}d{\'e}ric and Maslovskaya, Sofya},
  journal={arXiv:2406.14270},
  year={2024}
}

@article{hallinan2025inverse,
  title={Inverse Optimal Control for Passive Network Systems},
  author={Hallinan, Liam and Watson, Jeremy D and Lestas, Ioannis},
  journal={IEEE Transactions on Automatic Control},
  year={2025},
  publisher={IEEE}
}

@book{vershynin2018hdp,
  author    = {Vershynin, Roman},
  title     = {High-Dimensional Probability: An Introduction with Applications in Data Science},
  publisher = {Cambridge University Press},
  address   = {Cambridge},
  year      = {2018}
}

@article{richardson1911approx,
  author  = {Richardson, Lewis Fry},
  title   = {The Approximate Arithmetical Solution by Finite Differences of Physical Problems Involving Differential Equations, with an Application to the Stresses in a Masonry Dam},
  journal = {Philosophical Transactions of the Royal Society of London. Series A},
  volume  = {210},
  number  = {459--470},
  pages   = {307--357},
  year    = {1911},
  doi     = {10.1098/rsta.1911.0009}
}

@article{joyce1971survey,
  author  = {Joyce, D. C.},
  title   = {Survey of Extrapolation Processes in Numerical Analysis},
  journal = {SIAM Review},
  volume  = {13},
  number  = {4},
  pages   = {435--490},
  year    = {1971},
  doi     = {10.1137/1013092}
}

@article{dean2020sample,
  title={On the sample complexity of the linear quadratic regulator},
  author={Dean, Sarah and Mania, Horia and Matni, Nikolai and Recht, Benjamin and Tu, Stephen},
  journal={Foundations of Computational Mathematics},
  volume={20},
  number={4},
  pages={633--679},
  year={2020},
  doi={10.1007/s10208-019-09426-y}
}

@inproceedings{mania2019certainty,
  title={Certainty equivalence is efficient for linear quadratic control},
  author={Mania, Horia and Tu, Stephen and Recht, Benjamin},
  booktitle={Advances in Neural Information Processing Systems},
  volume={32},
  year={2019}
}

@article{diamond2016cvxpy,
  author  = {Diamond, Steven and Boyd, Stephen},
  title   = {{CVXPY}: A Python-Embedded Modeling Language for Convex Optimization},
  journal = {Journal of Machine Learning Research},
  year    = {2016},
  volume  = {17},
  number  = {83},
  pages   = {1--5},
}

@misc{goulart2024clarabel,
  author       = {Goulart, Paul J. and Chen, Yuwen},
  title        = {Clarabel: An Interior-Point Solver for Conic Programs with Quadratic Objectives},
  year         = {2024},
  eprint       = {2405.12762},
  archivePrefix= {arXiv},
  primaryClass = {math.OC},
}

\end{document}